\documentclass[11pt, twoside, leqno]{article}

\usepackage{amssymb}
\usepackage{amsmath}
\usepackage{amsthm}
\usepackage{color}
\usepackage{mathrsfs}
\usepackage{txfonts}

\usepackage{indentfirst}

\allowdisplaybreaks

\def\XXint#1#2#3{{\setbox0=\hbox{$#1{#2#3}{\int}$ }
\vcenter{\hbox{$#2#3$ }}\kern-.6\wd0}}

\def\rr{{\mathbb R}}
\def\ee{{\mathbb E}}
\def\rn{{\mathbb{R}^n}}
\def\zz{{\mathbb Z}}
\def\cc{{\mathbb C}}
\def\ss{{\mathbb S}}
\def\nn{{\mathbb N}}

\def\cf{{\mathcal F}}

\def\cm{{\mathcal M}}
\def\ct{{\mathcal T}}

\def\MP{{\mathbb{P}}}
\def\WP{WL_{\varphi}(\Omega)}
\def\V{\varphi}
\def\D{(0,\infty)}

\def\WH{WH_{\varphi}^s(\Omega)}
\def\cwp{W\mathcal{L}_{\V}(\Omega)}

\def\fz{\infty }

\def\lf{\left}
\def\r{\right}
\def\hs{\hspace{0.25cm}}
\def\ls{\lesssim}

\def\noz{\nonumber}

\def\dis{\displaystyle}

\newtheorem{theorem}{Theorem}[section]
\newtheorem{lemma}[theorem]{Lemma}

\theoremstyle{definition}
\newtheorem{remark}[theorem]{Remark}
\newtheorem{assumption}[theorem]{Assumption}
\newtheorem{definition}[theorem]{Definition}
\renewcommand{\appendix}{\par
   \setcounter{section}{0}%
   \setcounter{subsection}{0}%
   \setcounter{subsubsection}{0}%
   \gdef\thesection{\@Alph\c@section}%
   \gdef\thesubsection{\@Alph\c@section.\@arabic\c@subsection}%
   \gdef\theHsection{\@Alph\c@section.}%
   \gdef\theHsubsection{\@Alph\c@section.\@arabic\c@subsection}%
   \csname appendixmore\endcsname
 }

\numberwithin{equation}{section}

\begin{document}

\arraycolsep=1pt

\title{\bf\Large  Atomic Characterizations of Weak Martingale Musielak--Orlicz Hardy
Spaces and Their Applications
\footnotetext{\hspace{-0.35cm} 2010 {\it
Mathematics Subject Classification}. Primary  60G42;
Secondary 60G46, 42B25, 42B35.
\endgraf {\it Key words and phrases.} probability space, atomic characterization,
weak martingale Musielak--Orlicz Hardy space, sublinear operator.
\endgraf This project is supported by the National
Natural Science Foundation of China
(Grant Nos. 11571039, 11726621, 11671185 and 11761131002). }}
\author{Guangheng Xie and Dachun Yang\,\footnote{Corresponding author/{\color{red}October 22, 2018}/Final version.}}
\date{}
\maketitle

\vspace{-0.8cm}

\begin{center}
\begin{minipage}{13cm}
{\small {\bf Abstract}\quad
Let $(\Omega,\mathcal{F},\mathbb{P})$ be a probability space and
$\varphi:\ \Omega\times[0,\infty)\to [0,\infty)$ be a Musielak--Orlicz
function. In this article, the authors establish the atomic characterizations of weak martingale Musielak--Orlicz Hardy spaces $WH_{\varphi}^s(\Omega)$, $WH_{\varphi}^M(\Omega)$, $WH_{\varphi}^S(\Omega)$, $WP_{\varphi}(\Omega)$ and $WQ_{\varphi}(\Omega)$.
Using these atomic characterizations, the authors then obtain the boundedness of
sublinear operators from weak martingale Musielak--Orlicz Hardy spaces to
weak Musielak--Orlicz spaces, and some martingale inequalities which further
clarify the relationships among $WH_{\varphi}^s(\Omega)$, $WH_{\varphi}^M(\Omega)$, $WH_{\varphi}^S(\Omega)$, $WP_{\varphi}(\Omega)$ and $WQ_{\varphi}(\Omega)$.
All these results improve and generalize the corresponding results on
weak martingale Orlicz--Hardy spaces.
Moreover, the authors also improve all the known results on weak
martingale Musielak--Orlicz Hardy spaces. In particular, both the boundedness
of sublinear operators and the martingale inequalities, for the weak weighted
martingale Hardy spaces as well as for the weak weighted martingale Orlicz--Hardy spaces, are new.
}
\end{minipage}
\end{center}

\vspace{0.2cm}

\section{Introduction\label{s1}}
As is well known, the classical weak Hardy spaces naturally appear
when studying the boundedness of operators in critical cases. Indeed, Fefferman
and Soria \cite{fs87} originally introduced the weak Hardy space $WH^1(\rn)$
and proved in \cite[Theorem 5]{fs87} that some Calder\'on--Zygmund operators
are bounded from $WH^1(\rn)$ to weak Lebesgue spaces $WL^1(\rn)$.
It should also point out that Fefferman et al. \cite{frs} proved that the weak
Hardy spaces are the intermediate spaces of Hardy spaces in the real interpolation method.

Recently, various martingale Hardy spaces were investigated; see, for example,
Weisz \cite{W90,W94,W16}, Ho \cite{H,H16}, Nakai et al. \cite{mns12,ns,nss}, Sadasue \cite{S}
and Jiao et al. \cite{jzww,xjy18} for various different martingale Hardy spaces and their applications.
Moreover, the theory of weak martingale Hardy spaces has also been developed rapidly.
The weak Hardy spaces consisting of Vilenkin martingales were originally studied by
Weisz \cite{w98} and then fully generalized by Hou and Ren \cite{hr06}.
Inspired by these, Jiao et al. \cite{j,jwp15} and Liu et al. \cite{lzp17,lz17}
investigated the weak martingale Orlicz--Hardy spaces associated with concave
functions. Zhou et al. \cite{zwj} introduced the weak martingale Orlicz--Karamata--Hardy spaces
associated with concave functions and established their atomic characterizations.

On another hand, as a generalization of the Orlicz space and the weighted Lebesgue space,
the Musielak--Orlicz spaces prove very useful in dealing with some problems of analysis,
probability and partial differential equations
(see, for example, \cite{bgk,Yang2017}, \cite{xjy18,xwyj}, \cite{acgy} and their references). Very recently, Yang \cite{Y17} introduced the weak martingale Musielak--Orlicz Hardy spaces
which are a generalization of weak martingale Orlicz--Hardy spaces (see, for example, \cite{jwp15}).
Moreover, Yang \cite{Y17} also established the atomic characterizations
of weak martingale Musielak--Orlicz Hardy spaces and the boundedness of sublinear operators
from weak martingale Musielak--Orlicz Hardy spaces to weak Musielak--Orlicz spaces.

Let $(\Omega,\mathcal{F},\mathbb{P})$ be a probability space.
A function $\varphi:\ \Omega\times[0,\infty)\to[0,\infty)$
is called a \emph{Musielak--Orlicz} \emph{function} if the function $\varphi(\cdot,t)$ is a measurable
function for any given $t\in[0,\infty)$, and the function
$\varphi(x,\cdot):\ [0,\infty)\to[0,\infty)$ is an Orlicz function for any given $x\in\Omega$, namely,
$\varphi(x,\cdot)$ is non-decreasing, $\varphi(x,0)=0$, $\varphi(x,t)>0$ for any $t\in(0,\fz)$ and
$\lim_{t\to\infty}\varphi(x,t)=\infty$.
For any $p\in(0,\infty),$ a Musielak--Orlicz function $\varphi$ is said to
be of \emph{uniformly lower} (resp., \emph{upper}) \emph{type} $p$
if there exists a positive constant $C_{(p)},$ depending on \emph{p}, such that
\begin{align}\label{i1}
\varphi(x,st)\le C_{(p)}s^p\varphi(x,t)
\end{align}
for any $x\in\Omega$ and $t\in[0,\infty),$ $s\in(0,1)$ (resp., $s\in[1,\infty)$);
see \cite{Yang2017} for more details.

Recall that the following assumption is needed through \cite{Y17}.

\setcounter{theorem}{0}
\renewcommand{\thetheorem}{\arabic{section}.\Alph{theorem}}

\begin{assumption}\label{a1}
Let $\V$ be a Musielak--Orlicz function and let $\V$ be of uniformly lower type $p\in(0,1]$
and of uniformly upper type $1$.
\end{assumption}

\renewcommand{\thetheorem}{\arabic{section}.\arabic{theorem}}

Observe that Assumption \ref{a1} is quite restrictive.
Indeed, for any given $p\in(1,\fz)$, if $\V(x,t):=t^p$ for any $x\in\Omega$ and $t\in\D$,
then $\V$ is of uniformly lower type $p$ and also of uniformly upper type $p$. However, in this case,
$\V$ is not of uniformly upper type 1.
Thus, under Assumption \ref{a1}, all the results in \cite{Y17} can not cover
the corresponding results on weak Lebesgue spaces $WL_p(\Omega)$ with any given $p\in(1,\fz)$ in \cite{w98,hr06}.

On another hand, Jiao et al. \cite{jwp15} studied weak martingale Orlicz--Hardy spaces under
the following assumption.
For any $\ell\in(0,\infty)$, let $\mathcal{G}_{\ell}$ be the set of all Orlicz functions $\Phi$
satisfying that $\Phi$ is of lower type $\ell$ and of upper type $1$
(see, for example, \cite{jwp15,mns12}). Let $\Phi$ be a concave function and
$\Phi'$ its derivative function. Its lower index and its upper index of
$\Phi$ are defined, respectively, by setting
\begin{equation}\label{j}
p_{\Phi}:=\inf_{t\in\D}\frac{t\Phi'(t)}{\Phi(t)} \quad
\mbox{and} \quad q_{\Phi}:=\sup_{t\in\D}\frac{t\Phi'(t)}{\Phi(t)}.
\end{equation}
All the results in \cite{jwp15} need the assumptions that
$\Phi\in \mathcal{G}_{\ell}$ for some $\ell\in(0,1]$ and $q_{\Phi^{-1}}\in\D$,
here $\Phi^{-1}$ denotes the inverse function of $\Phi$.
Observe that, when $\V(x,t):=\Phi(t)$ for any $x\in\Omega$ and $t\in\D$,
$\V$ satisfies Assumption \ref{a1} if and only if
$\Phi\in \mathcal{G}_{\ell}$ for some $\ell\in(0,1]$.

The first motivation of this article is to weaken Assumption \ref{a1} of \cite{Y17} and
to remove the unnecessary assumption $q_{\Phi^{-1}}\in\D$ of \cite{jwp15}. Indeed,
instead of Assumption \ref{a1}, in this article, we always make the following assumption.

\setcounter{theorem}{0}

\begin{assumption}\label{a1-new}
Let $\V$ be a Musielak--Orlicz function, and let $\V$ be of uniformly
lower type $p_{\V}^{-}$ for some $p_{\V}^{-}\in(0,\fz)$
and of uniformly upper type $p_{\V}^{+}$ for some $p_{\V}^{+}\in(0,\fz)$.
\end{assumption}

In this article, under Assumption \ref{a1-new}, we first establish the atomic
characterizations of weak martingale Musielak--Orlicz Hardy spaces $WH_{\varphi}^s(\Omega)$,
$WH_{\varphi}^M(\Omega)$, $WH_{\varphi}^S(\Omega)$, $WP_{\varphi}(\Omega)$ and $WQ_{\varphi}(\Omega)$.
Using these atomic characterizations, we then obtain the boundedness of
sublinear operators from weak martingale Musielak--Orlicz Hardy spaces to
weak Musielak--Orlicz spaces, and some martingale inequalities which
further clarify the relationships among $WH_{\varphi}^s(\Omega)$, $WH_{\varphi}^M(\Omega)$,
$WH_{\varphi}^S(\Omega)$, $WP_{\varphi}(\Omega)$ and $WQ_{\varphi}(\Omega)$.
All these results improve and generalize the corresponding
results on weak martingale Orlicz--Hardy spaces (see \cite{jwp15}).
Moreover, we also improve all the results on weak
martingale Musielak--Orlicz Hardy spaces in \cite{Y17}.
In particular, both the boundedness of sublinear operators and
the martingale inequalities, for the weak weighted
martingale Hardy spaces as well as for the weak weighted martingale Orlicz--Hardy spaces, are new.

To be precise, this article is organized as follows.

In Section \ref{s2}, we first recall some notation and notions on Musielak--Orlicz
functions, weak Musielak--Orlicz spaces and weak martingale Musielak--Orlicz Hardy
spaces. Then we introduce various weak atomic martingale Musielak--Orlicz Hardy spaces.

Section \ref{s3} is devoted to establishing the atomic characterizations of spaces
$WH_{\varphi}^s(\Omega)$, $WH_{\varphi}^M(\Omega)$, $WH_{\varphi}^S(\Omega)$,
$WP_{\varphi}(\Omega)$ and $WQ_{\varphi}(\Omega)$
(see Theorems \ref{thm-atom}, \ref{thm-apq} and \ref{thm-atomSM} below).
The above five weak martingale Musielak--Orlicz Hardy spaces contain weak
weighted martingale Hardy spaces, weak martingale Orlicz--Hardy spaces
in \cite{jwp15} and weak variable martingale Hardy spaces as special cases
(see Remark \ref{rmk-sc} below for more details). Recall that, even for
weak martingale Hardy spaces in \cite{w98,hr06}, only the $\infty$-atomic
characterization is known. However, we establish the $q$-atomic
characterizations for any $q\in(\max\{p_{\V}^{+},1\},\fz]$ in this article,
where $p_{\V}^{+}$ is the uniformly upper type index of $\V$.
Moreover, in \cite{jwp15} for weak martingale Orlicz--Hardy spaces and \cite{Y17}
for weak martingale Musielak--Orlicz Hardy spaces, the results of
atomic characterizations need the index $p_{\V}^{+}=1$. Differently from \cite{jwp15,Y17},
we allow $p_{\V}^{+}\in\D$ in Theorems \ref{thm-atom}, \ref{thm-apq}
and \ref{thm-atomSM} below. So, the classical argument used in the proof of
\cite[Theorem 1]{hr06} and \cite[Theorem 2.1]{jwp15}
does not work here anymore. Via using some ideas from the proofs of \cite[Theorem 3.5]{lyj16}
and constructing some appropriate atoms, we overcome this difficulty; see the proof of
Theorems \ref{thm-atom} and \ref{thm-atomSM}. Moreover, our atomic characterizations
of weak martingale Musielak--Orlicz Hardy spaces cover weak variable martingale Hardy spaces,
weak weighted martingale Hardy spaces and weak weighted martingale Orlicz--Hardy spaces,
which are also new (see Remarks \ref{rmk-cac} and \ref{rmk-cc} below).

In Section \ref{s4}, we study the boundedness of sublinear
operators on weak martingale Musielak--Orlicz Hardy
spaces. Recall that, for a martingale space $X$ and a measurable function space $Y$, an operator
$T:\ X\to Y$ is called a \emph{sublinear operator} if, for any $f,$ $g\in X$ and $c\in\rr$,
$$\lf|T(f+g)\r|\le \lf|T(f)\r|+\lf|T(g)\r|\quad\mathrm{and}\quad |T(cf)|\le |c||T(f)|.$$
The boundedness of sublinear operators from the weak martingale Hardy spaces
to weak Lebesgue spaces was studied in \cite{hr06,w98}, and then
from the weak martingale Orlicz--Hardy spaces
to weak Orlicz spaces in \cite{jwp15}. All these results need the assumption
that sublinear operators $T$ are bounded on
$L^q(\Omega)$ for some $q\in[1,2]$ or some $q\in[1,\fz)$.
Particularly, in \cite[Theorem 4.2]{Y17}, Yang also gave some sufficient
conditions for a sublinear operator $T$ to be bounded from the weak
martingale Musielak--Orlicz Hardy spaces to weak Musielak--Orlicz spaces.
In what follows, for any measurable set $E\subseteq\Omega$ and $t\in[0,\fz)$, let
$\varphi(E,t):=\int_{E}\varphi(x,t)\,d\mathbb{P}$. The following assumption on $\V$ is
needed in Yang \cite[Theorems 4.2 through 4.5]{Y17}.

\setcounter{theorem}{1}
\renewcommand{\thetheorem}{\arabic{section}.\Alph{theorem}}

\begin{assumption}\label{a2}
\begin{itemize}
\item[(i)] Let $T$ be a sublinear operator bounded on $L^2(\Omega)$.
\item[(ii)] Let $\V$ be a Musielak--Orlicz function satisfying Assumption \ref{a1}
and there exist two positive constants $B$ and $D$ such that,
for any measurable subset $E\subseteq\Omega$, $x\in\Omega$ and $t\in\D$,
\begin{align}\label{y17}
B\V(x,t)\MP(E)\le\V(E,t)\le D\V(x,t)\MP(E).
\end{align}
\end{itemize}
\end{assumption}

\renewcommand{\thetheorem}{\arabic{section}.\arabic{theorem}}

Observe that \eqref{y17} is also quite restrictive.
Indeed, using \eqref{y17} with $E=\Omega$,
we find that, for any $x\in\Omega$ and $t\in\D$,
$$\frac{1}{D}\V(\Omega,t)\le\V(x,t)\le\frac{1}{B}\V(\Omega,t).$$
Thus, Assumption \ref{a2} requires $\V$ to be essentially
an Orlicz function. Moreover, \cite[Theorems 4.2 through 4.5]{Y17}
do not cover the very important case, namely, the weighted case.

Observe that all these assumptions for the boundedness of sublinear operators
used in \cite{hr06,w98,jwp15,Y17} ensure that $T$ is
bounded from some martingale Hardy spaces to some Lebesgue spaces,
which, together with the fact that Musielak--Orlicz functions
unify Orlicz functions and weights,
motivates us to introduce the following assumption.

\setcounter{theorem}{1}

\begin{assumption}\label{a2-new}
Let $\V$ be a Musielak--Orlicz function satisfying Assumption \ref{a1-new}.
Let $T$ be a sublinear operator and satisfy one of the following:
\begin{itemize}
\item[(i)] for some given $q\in(p_{\V}^{+},\fz)$, $T$ is bounded from  the weighted Hardy space
$H_q^s(\Omega,\V(\cdot,t)\,d\MP)$ to the weighted Lebesgue space $L^q(\Omega,\V(\cdot,t)\,d\MP)$;
\item[(ii)] for some given $q\in(p_{\V}^{+},\fz)$, $T$ is bounded from  the weighted Hardy space
$H_q^S(\Omega,\V(\cdot,t)\,d\MP)$ to the weighted Lebesgue space $L^q(\Omega,\V(\cdot,t)\,d\MP)$;
\item[(iii)] for some given $q\in(p_{\V}^{+},\fz)$, $T$ is bounded from  the weighted Hardy space
$H_q^M(\Omega,\V(\cdot,t)\,d\MP)$ to the weighted Lebesgue space $L^q(\Omega,\V(\cdot,t)\,d\MP)$.
\end{itemize}
(See Section \ref{s2} for the definitions of these spaces.)
\end{assumption}

In Section \ref{s4} of this article, under Assumption \ref{a2-new},
we obtain the boundedness of sublinear operators from $WH_{\varphi}^s(\Omega)$
(resp., $WH_{\varphi}^M(\Omega)$, $WH_{\varphi}^S(\Omega)$, $WP_{\varphi}(\Omega)$ or $WQ_{\varphi}(\Omega)$) to $\WP$;
see Theorems \ref{thm-sub}, \ref{thm-subpq} and \ref{thm-subSM} below.
Particularly, we obtain the same results
as in \cite{Y17} via replacing Assumption \ref{a2}
by Assumption \ref{a2-new}.

Observe that Assumption \ref{a2-new} is much weaker than Assumption \ref{a2}. Indeed,
Assumption \ref{a1-new} is weaker than Assumption \ref{a1} and the assumption \eqref{y17}
in Assumption \ref{a2} does not needed in Assumption \ref{a2-new}.
Moreover, under Assumption \ref{a2},
we find that $p_{\V}^{+}=1$ and the weighted Hardy space $H_2^s(\Omega,\V(\cdot,t)\,d\MP)$
[resp., $H_q^S(\Omega,\V(\cdot,t)\,d\MP)$ or $H_q^M(\Omega,\V(\cdot,t)\,d\MP)$]
becomes the martingale Hardy space $H_2^s(\Omega)$ [resp., $H_2^S(\Omega)$ or $H_2^M(\Omega)$],
which, together with the boundedness of $T$ on $L^2(\Omega)$ and
the boundedness of the operator $s$ (resp., $S$ or $M$) on $L^2(\Omega)$
(see, for example, \cite[Proposition 2.6 and Theorems 2.11 and 2.12]{W94}),
further implies that Assumption \ref{a2-new} holds true. Thus, compared
with Assumption \ref{a2}, Assumption \ref{a2-new} is much weaker.
In particular, Theorems \ref{thm-sub}, \ref{thm-subpq} and \ref{thm-subSM}
of this article indeed improve \cite[Theorems 4, 5 and 6]{hr06},
\cite[Theorem 3.1 and Remark 3.2]{jwp15} and \cite[Theorems 4.2, 4.3 and 4.4]{Y17},
respectively (see Remark \ref{rmk-su} below for more details).

Also, in this section, using Theorems \ref{thm-sub} and \ref{thm-subpq},
we obtain some martingale inequalities among the spaces $WH_{\varphi}^s(\Omega)$, $WH_{\varphi}^M(\Omega)$, $WH_{\varphi}^S(\Omega)$, $WP_{\varphi}(\Omega)$ and
$WQ_{\varphi}(\Omega)$, which further
clarify the relations among these spaces, in Theorem \ref{thm-mi} below.
Moreover, Theorem \ref{thm-mi} generalizes and improves the corresponding results
on weak martingale Orlicz--Hardy spaces in \cite[Theorem 3.3]{jwp15}
(see Remark \ref{rmk=mi} below for the details).

In Section \ref{s5}, the last section of this article, we obtain some
bounded convergence theorems and dominated convergence theorems on
weak Musielak--Orlicz spaces $\WP$
(see Theorems \ref{Thm-bct} and \ref{thm-dct} below), which are of independent interest.

Finally, we make some conventions on notation used in this article. Throughout the article,
we always let $\mathbb{N}:=\{1,2,\ldots\}$, $\mathbb{Z}_{+}:=\mathbb{N}\cup \{0\}$
and $C$ denote a positive constant, which may vary from line to line.
We use the symbol $f\lesssim g$ to denote that there exists a positive constant $C$
such that $f\le Cg.$ The symbol $f\sim g$ is used as an abbreviation
of $f\lesssim g\lesssim f$. We also use the following
convention: If $f\le Cg$ and $g=h$ or $g\le h$, we then write $f\ls g\sim h$
or $f\ls g\ls h$, \emph{rather than} $f\ls g=h$
or $f\ls g\le h$. For any subset $E$ of $\Omega$,
denote $\mathbf{1}_E$ by its \emph{characteristic function}.
For any $p\in[1,\fz]$, let $p'$ denote the conjugate number of $p$, namely, $1/p+1/p'=1$.

\section{Preliminaries\label{s2}}

In this section, we first recall some notation and notions on Musielak--Orlicz
functions, weak Musielak--Orlicz spaces and weak martingale Musielak--Orlicz Hardy
spaces and then we introduce various weak atomic martingale Musielak--Orlicz Hardy spaces.

\begin{definition}
Let $\varphi$ be a Musielak--Orlicz function. The \emph{weak Musielak--Orlicz space}
$WL_{\varphi}(\Omega)$ is defined to be the set of all measurable functions $f$
such that
$$\lf\|f\r\|_{WL_{\varphi}(\Omega)}:=\inf\lf\{\lambda\in(0,\fz):\
\sup_{\alpha\in(0,\fz)}\varphi\lf(\lf\{x\in\Omega:\
\lf|f(x)\r|>\alpha\r\},\frac{\alpha}{\lambda}\r)\leq1\r\}<\fz.$$
\end{definition}
Let $p\in\D$ and $\Phi$ be an Orlicz function. If $\V(x,t):=t^p$ or $\Phi(t)$
for any $x\in\Omega$ and $t\in\D$, then $\WP$
becomes weak Lebesgue spaces $WL_p(\Omega)$  (see, for example, \cite{w98})
or weak Orlicz space $WL_{\Phi}(\Omega)$ (see, for example, \cite{jwp15}),
here and hereafter, $WL_p(\Omega)$ (resp., $WL_{\Phi}(\Omega)$) denotes
the set of all measurable functions $f$ on $\Omega$ such that
$$\|f\|_{WL_p(\Omega)}:=\sup_{\alpha\in\D}\alpha\lf[\MP\lf(\{x\in\Omega:\ |f(x)|>\alpha\}\r)\r]^{\frac1p}<\fz$$
$$\lf[\mbox{resp.},\quad\|f\|_{WL_{\Phi}(\Omega)}
:=\inf\lf\{\lambda\in\D:\ \sup_{\alpha\in\D}\Phi
\lf(\frac{\alpha}{\lambda}\r)\MP\lf(\{x\in\Omega:\ |f(x)|>\alpha\}\r)\le1\r\}<\fz\r].$$

\begin{remark}\label{tri}
If a Musielak--Orlicz function $\V$ is of uniformly upper type $p_{\V}^{+}$
for some $p_{\V}^{+}\in\D$, then there exists a positive constant $C$ such that,
for any measurable functions $f$ and $g$,
$$\lf\|f+g\r\|_{WL_{\varphi}(\Omega)}\le C\lf[\lf\|f\r\|_{WL_{\varphi}(\Omega)}
+\lf\|g\r\|_{WL_{\varphi}(\Omega)}\r].$$
Indeed, by the uniformly upper type $p_{\V}^{+}$ property of $\V$, we find that,
for any $\lambda\in\D$,
\begin{align*}
&\sup_{\alpha\in(0,\fz)}\varphi\lf(\lf\{x\in\Omega:\
\lf|f(x)+g(x)\r|>\alpha\r\},\frac{\alpha}{\lambda}\r)\\
&\hs\lesssim \sup_{\alpha\in(0,\fz)}\varphi\lf(\lf\{x\in\Omega:\
\lf|f(x)\r|>\frac{\alpha}{2}\r\},\frac{\alpha}{2\lambda}\r)
+\sup_{\alpha\in(0,\fz)}\varphi\lf(\lf\{x\in\Omega:\
\lf|g(x)\r|>\frac{\alpha}{2}\r\},\frac{\alpha}{2\lambda}\r)\\
&\hs\sim \sup_{\alpha\in(0,\fz)}\varphi\lf(\lf\{x\in\Omega:\
\lf|f(x)\r|>\alpha\r\},\frac{\alpha}{\lambda}\r)
+\sup_{\alpha\in(0,\fz)}\varphi\lf(\lf\{x\in\Omega:\
\lf|g(x)\r|>\alpha\r\},\frac{\alpha}{\lambda}\r).
\end{align*}
Then the above claim follows immediately.
\end{remark}

\begin{remark}\label{rem-uni}
Obviously, if $\varphi$ is both of uniformly lower
type $p_1$ and of uniformly upper type $p_2$,
then $p_1\le p_2$. Moreover, if $\varphi$ is of uniformly
lower (resp., upper) type $p$, then,
it is also of uniformly lower (resp., upper) type
$\widetilde{p}$ for any $\widetilde{p}\in(0,p)$
[resp., $\widetilde{p}\in(p,\fz)$].
\end{remark}

\begin{remark}\label{rem-p}
Let $\V$ be a Musielak--Orlicz function satisfying Assumption \ref{a1-new}.
If there exist an Orlicz function $\Phi$ and
two positive constants $B$ and $D$ such that, for any $x\in\Omega$ and $t\in\D$,
$$B\Phi(t)\le\V(x,t)\le D\Phi(t),$$
then, for any $f\in\WP$, $\|f\|_{\WP}\sim \|f\|_{WL_{\Phi}(\Omega)}$
with the positive equivalence constants independent of $f$.
\end{remark}

Let $\{\mathcal{F}_n\}_{n\in\mathbb{Z}_{+}}$ be an increasing
sequence of sub-$\sigma$-algebras of $\mathcal{F}$ and let
$\{\mathbb{E}_n\}_{n\in\mathbb{Z}_{+}}$ be the
associated conditional expectations. The weight we consider in this article
are\emph{ special weights} with respect to
$(\Omega,\mathcal{F},\mathbb{P},\{\mathcal{F}_n\}_{n\in\mathbb{Z}_{+}})$, that is,
the martingale generated by $\V$, where $\V$ is a
Musielak--Orlicz function, which is strictly positive, and
$\sup_{t\in\D}\int_{\Omega}\V(x,t)\,d\MP<\fz$.
More precisely, let $\varphi(\cdot,t):=\{\varphi_n(\cdot,t)\}_{n\in\mathbb{Z}_{+}}$
be the martingale generated by $\varphi(\cdot,t)$ for any $t\in(0,\infty).$
For simplicity, we still use $\varphi(\cdot,t)$ to denote the martingale $\varphi(\cdot,t):=\{\varphi_n(\cdot,t)\}_{n\in\mathbb{Z}_{+}}$.

The following weighted condition is due to Izumisawa and Kazamaki \cite{IK77}.

\begin{definition}\label{def-ap}
Let $q\in[1,\fz)$. A positive Musielak--Orlicz function
$\V:\ \Omega\times[0,\fz)\to[0,\fz)$ is said to
satisfy the \emph{uniformly} \emph{$A_q(\Omega)$ condition},
denoted by $\V\in\mathbb{A}_q(\Omega)$,
if there exists a positive constant $K$ such that,
when $q\in(1,\fz)$,
$$\sup_{t\in\D}\ee_n(\V)(\cdot,t)
\lf[\ee_n\lf(\V^{-\frac1{q-1}}\r)(\cdot,t)\r]^{q-1}\le K
\quad \MP{\text-}{\rm almost\ \ everywhere},\quad \forall\,n\in\mathbb{Z}_{+}$$
and, when $q=1$,
$$\sup_{t\in\D}\ee_n(\V)(\cdot,t)\frac{1}{\V(\cdot,t)}\le K
\quad \MP{\text-}{\rm almost\ \ everywhere},\quad \forall\,n\in\mathbb{Z}_{+}.$$
A positive Musielak--Orlicz funtion $\V$ is said to belong to
$\mathbb{A}_{\fz}(\Omega)$ if $\V\in \mathbb{A}_q(\Omega)$ for some $q\in[1,\fz)$.
\end{definition}

The following $\ss$ condition arises naturally when dealing with the weighted
martingale inequalities. We refer to Dol\'eans-Dade and Meyer \cite{DM78} and
Bonami and L\'epingle \cite{bl79} for more details.

\begin{definition}\label{def-s}
Let $t\in[0,\infty)$. The martingale $\varphi(\cdot,t)
:=\{\varphi_n(\cdot,t)\}_{n\in\mathbb{Z}_{+}}$
is said to satisfy the \emph{uniformly $\mathbb{S}$ condition},
denoted by $\varphi\in \mathbb{S},$
if there exists a positive constant $K$ such that,
for any $n\in\mathbb{N}$, $t\in(0,\infty)$ and almost every $x\in\Omega,$
\begin{align}\label{WS}
\frac{1}{K}\varphi_{n-1}(x,t)\le\varphi_n(x,t)\le K\varphi_{n-1}(x,t).
\end{align}
The \emph{conditions $\mathbb{S}^{-}$} and \emph{$\mathbb{S}^{+}$} denote two parts
of $\mathbb{S}$ satisfying only the left or the right hand sides of the
preceding inequalities, respectively.
\end{definition}
Let $w$ be a special weight on $\Omega$ and $\V(x,t):=w(x)$
for any $x\in\Omega$ and $t\in\D$.
Then Definitions \ref{def-ap} and \ref{def-s} go
back to the original weighted definition \cite{DM78,IK77}.

Denote by $\mathcal{M}$
the set of all martingales $f:=(f_n)_{n\in\mathbb{Z}_{+}}$ related to
$\{\mathcal{F}_n\}_{n\in\mathbb{Z}_{+}}$ such that $f_0=0.$
For any $f\in\mathcal{M},$ denote its \emph{martingale difference sequence}
by $\{d_nf\}_{n\in\nn}$, where $d_nf:=f_n-f_{n-1}$ for any $n\in\mathbb{N}$.
Then the \emph{maximal functions} $M_n(f)$ and $M(f)$,
the \emph{quadratic variations} $S_n(f)$ and $S(f)$,
and the \emph{conditional quadratic variations} $s_n(f)$
and $s(f)$ of the martingale $f$ are defined, respectively, by setting
\begin{align*}
M_n(f):=\sup_{0\le i\le n}|f_i|,\quad M(f):=\sup_{n\in\mathbb{Z}_{+}}|f_n|,
\end{align*}
\begin{align*}
S_n(f):=\left(\sum_{i=1}^n\left|d_if\right|^2\right)^{\frac{1}{2}},\quad S(f):=\left(\sum_{i=1}^{\infty}\left|d_if\right|^2\right)^{\frac{1}{2}},
\end{align*}
\begin{align*}
s_n(f):=\left(\sum_{i=1}^n\mathbb{E}_{i-1}\left|d_if\right|^2\right)^{\frac{1}{2}}
\quad \mathrm{and}\quad
s(f):=\left(\sum_{i=1}^{\infty}\mathbb{E}_{i-1}\left|d_if\right|^2\right)^{\frac{1}{2}}.
\end{align*}
Let $\Lambda$ be the collection of all sequences
$(\lambda_n)_{n\in\mathbb{Z}_{+}}$ of
nondecreasing, nonnegative and \emph{adapted functions}
[namely, for any $n\in\mathbb{Z}_{+}$,
$\lambda_n$ is $\mathcal{F}_n$ measurable].
Let $\lambda_{\infty}:=\lim_{n\to\infty}\lambda_n.$
For any $f\in\mathcal{M},$ let
\begin{align*}
\Lambda[WP_{\varphi}](f):=\lf\{(\lambda_n)_{n\in\mathbb{Z}_{+}}
\in\Lambda:\ |f_n|\le\lambda_{n-1}\ \ (n\in\mathbb{N}),\ \ \lambda_{\infty}\in \WP\r\}
\end{align*}
and
\begin{align*}
\Lambda[WQ_{\varphi}](f):=\lf\{(\lambda_n)_{n\in\mathbb{Z}_{+}}
\in\Lambda:\ S_n(f)\le\lambda_{n-1}\ \ (n\in\mathbb{N}),\ \ \lambda_{\infty}\in \WP\r\}.
\end{align*}

\begin{definition}\label{def-spaces}
Let $\varphi$ be a Musielak--Orlicz function.
The \emph{weak martingale Musielak--Orlicz Hardy spaces}
$WH_{\varphi}^M(\Omega)$, $WH_{\varphi}^S(\Omega)$,
$WH_{\varphi}^s(\Omega)$, $WP_{\varphi}(\Omega)$
and $WQ_{\varphi}(\Omega)$ are, respectively,
defined as follows:
\begin{align*}
WH_{\varphi}^M(\Omega):=\left\{f\in\mathcal{M}:\ \|f\|_{WH_{\varphi}^M(\Omega)}:=\lf\|M(f)\r\|_{\WP}<\infty\right\},
\end{align*}
\begin{align*}
WH_{\varphi}^S(\Omega):=\left\{f\in\mathcal{M}:\ \|f\|_{H_{\varphi}^S}
:=\lf\|S(f)\r\|_{\WP}<\infty\right\},
\end{align*}
\begin{align*}
WH_{\varphi}^s(\Omega):=\left\{f\in\mathcal{M}
:\ \|f\|_{H_{\varphi}^s}:=\lf\|s(f)\r\|_{\WP}<\infty\right\},
\end{align*}
\begin{align*}
WP_{\varphi}(\Omega):=\left\{f\in\mathcal{M}:\ \|f\|_{WP_{\varphi}(\Omega):=\inf_{(\lambda_n)_{n\in\mathbb{Z}_{+}}
\in\Lambda[WP_{\varphi}(\Omega)]}}\|\lambda_{\infty}\|_{\WP}<\infty\right\}
\end{align*}
and
\begin{align*}
WQ_{\varphi}(\Omega):=\left\{f\in\mathcal{M}:\ \|f\|_{WQ_{\varphi}(\Omega)
:=\inf_{(\lambda_n)_{n\in\mathbb{Z}_{+}}\in\Lambda[WQ_{\varphi}(\Omega)]}}
\|\lambda_{\infty}\|_{\WP}<\infty\right\}.
\end{align*}
\end{definition}

\begin{remark}\label{rmk-sc}
Several known weak martingale Hardy spaces can be regarded as special cases of the
above five weak martingale Musielak--Orlicz Hardy spaces.
For example, let $p\in(0,\fz)$, $\Phi$ be an Orlicz function on $\D$, $w$ a weight
and $p(\cdot):\ \Omega\to [1,\fz]$ a measurable function. If
$\V(x,t):=t^p$, $\Phi(t)$, $t^{p(x)}$ or $w(x)\Phi(t)$ for
any $x\in\Omega$ and $t\in\D$, then the corresponding weak martingale Musielak--Orlicz
Hardy space becomes, respectively, the weak martingale Hardy space (see \cite{hr06,w98}),
the weak martingale Orlicz--Hardy space (see \cite{jwp15}),
the weak variable martingale Hardy space or the weak weighted martingale Orlicz--Hardy space.
\end{remark}

In what follows, for any $q\in[1,\infty]$, any measurable set $B\subseteq\Omega$
and any measurable function $f$ on $\Omega,$ let
$$\|f\|_{L_{\varphi}^q(B)}:=
\begin{cases}
\displaystyle\sup_{t\in(0,\infty)}\lf[\dfrac{1}{\varphi(B,t)}\dis\int_{\Omega}|f(x)|^q
\varphi(x,t)\,d\mathbb{P}(x)\r]^{1/q} \quad {\rm when}\ \ q\in[1,\infty),\\
\|f\|_{L^{\infty}(\Omega)} \hspace{5.296cm}{\rm when}\ \ q=\infty.
\end{cases}$$
Let $\mathcal{T}$ be the set of all stopping times related to
$\{\mathcal{F}_n\}_{n\in\mathbb{Z}_{+}}$.
For any $\nu\in\mathcal{T}$, let
$B_{\nu}:=\{x\in\Omega:\ \nu(x)<\fz\}$.
Now we introduce the notion of atoms associated with Musielak--Orlicz function.

\begin{definition}\label{def-atom}
Let $q\in(1,\infty]$ and $\varphi$ be a Musielak--Orlicz function. A measurable function $a$ is
called a \emph{$(\varphi,q)_s$-atom} if there exists a stopping time $\nu$ relative to $\{\mathcal{F}_n\}_{n\in\mathbb{Z}_{+}}$
($\nu$ is called the \emph{stopping time associated with $a$}) such that
\begin{enumerate}
\item[\rm{(i)}] $a_n:=\mathbb{E}_na=0$ if $\nu\geq n$,
\item[\rm{(ii)}] $\left\|s(a)\right\|_{L_{\varphi}^q(B_{\nu})}
          \le \|\mathbf{1}_{B_{\nu}}\|_{L^{\varphi}(\Omega)}^{-1}.$
\end{enumerate}
Similarly, \emph{$(\varphi,q)_S$-atom} and \emph{$(\varphi,q)_M$-atom} are
defined via replacing (ii) in the above definition by
$$\lf\|S(a)\r\|_{L_{\varphi}^q(B_{\nu})}
\le \|\mathbf{1}_{B_{\nu}}\|_{L^{\varphi}(\Omega)}^{-1},$$
respectively, by
$$\lf\|M(a)\r\|_{L_{\varphi}^q(B_{\nu})}
\le \|\mathbf{1}_{B_{\nu}}\|_{L^{\varphi}(\Omega)}^{-1}.$$
\end{definition}

Via $(\varphi,q)_s$-atoms, $(\varphi,q)_S$-atoms and $(\varphi,q)_M$-atoms,
we now introduce three weak atomic martingale Musielak--Orlicz Hardy spaces
$WH_{\rm at}^{\V,q,s}(\Omega)$, $WH_{\rm at}^{\V,q,S}(\Omega)$ and
$WH_{\rm at}^{\V,q,M}(\Omega)$, respectively, as follows.

\begin{definition}\label{def-ats}
Let $q\in(1,\fz]$ and $\V$ be a Musielak--Orlicz function. The \emph{weak atomic martingale
Musielak--Orlicz Hardy space} $WH_{\rm at}^{\V,q,s}(\Omega)$
[resp., $WH_{\rm at}^{\V,q,S}(\Omega)$ or $WH_{\rm at}^{\V,q,M}(\Omega)$]
is defined to be the space of all $f\in\mathcal{M}$
satisfying that there exist a sequence of $(\varphi,q)_s$-atoms
[resp., $(\varphi,q)_S$-atoms or $(\varphi,q)_M$-atoms] $\{a^k\}_{k\in\zz}$,
related to stopping times $\{\nu^k\}_{k\in\zz}$, and a positive constant $\widetilde{C}$,
independent of $f$, such that, for any $n\in\mathbb{Z}_{+}$,
\begin{align*}
\sum_{k\in\mathbb{Z}}\mu^ka_n^k=f_n \quad {\rm \MP-almost\ everywhere},
\end{align*}
where $\mu^k:=\widetilde{C}2^k\|\mathbf{1}_{B_{\nu^k}}\|_{L^{\varphi}(\Omega)}$
for any $k\in\zz$,
and
\begin{align*}
\|f\|_{WH_{\rm at}^{\V,q,s}(\Omega)}\quad\lf[\mbox{resp., }
\|f\|_{WH_{\rm at}^{\V,q,S}(\Omega)}\mbox{ or }\|f\|_{WH_{\rm at}^{\V,q,M}(\Omega)} \r]\\
:=\inf\lf\{\inf\lf[\lambda\in\D:\ \sup_{k\in\zz}
\V\lf(B_{\nu^k},\frac{2^k}{\lambda}\r)\le1\r]\r\}<\fz,
\end{align*}
where the first infimum is taken over all decompositions of $f$ as above.
\end{definition}

Let $p\in\D$ and $w$ be a special weight.
The \emph{weighted Lebesgue space $L^p(\Omega,w\,d\MP)$} is defined to be
the set of all measurable functions $f$ on $\Omega$ such that
$$\lf\|f\r\|_{L^p(\Omega,w\,d\MP)}
:=\lf[\int_{\Omega}\lf|f(x)\r|^pw(x)
\,d\mathbb{P}(x)\r]^{\frac1p}<\infty.$$
The \emph{weighted martingale Hardy spaces $H_p^s(\Omega,w\,d\MP)$, $H_p^S(\Omega,w\,d\MP)$
and $H_p^M(\Omega,w\,d\MP)$} are, respectively, defined as follows:
$$H_p^s(\Omega,w\,d\MP):=\lf\{f\in\mathcal{M}:\ \|s(f)\|_{L^p(\Omega,w\,d\MP)}<\fz\r\},$$
$$H_p^S(\Omega,w\,d\MP):=\lf\{f\in\mathcal{M}:\ \|S(f)\|_{L^p(\Omega,w\,d\MP)}<\fz\r\}$$
and $$H_p^M(\Omega,w\,d\MP):=\lf\{f\in\mathcal{M}:\ \|M(f)\|_{L^p(\Omega,w\,d\MP)}<\fz\r\}.$$
If $w\equiv1$, then the weighted Hardy space $H_p^s(\Omega,w\,d\MP)$
[resp., $H_p^S(\Omega,w\,d\MP)$ or $H_p^M(\Omega,w\,d\MP)$] becomes the classical
martingale Hardy space $H_p^s(\Omega)$ [resp., $H_p^S(\Omega)$ or $H_p^M(\Omega)$]
(see, for example, \cite[p.\,6]{W94}).

\section{Atomic characterizations\label{s3}}

In this section, we establish atomic characterizations of weak martingale
Musielak--Orlicz Hardy spaces
$WH_{\varphi}^s(\Omega)$, $WH_{\varphi}^M(\Omega)$, $WH_{\varphi}^S(\Omega)$,
$WP_{\varphi}(\Omega)$ and $WQ_{\varphi}(\Omega)$. We begin with the atomic
characterization of $WH_{\varphi}^s(\Omega)$.

\begin{theorem}\label{thm-atom}
Let $q\in\D$ and $\V$ be a Musielak--Orlicz function satisfying Assumption \ref{a1-new}.
If $q\in(\max\{p_{\V}^{+},1\},\fz]$, then $\WH=WH_{\rm at}^{\V,q,s}(\Omega)$
with equivalent quasi-norms.
\end{theorem}

\begin{proof}
We prove this theorem by two steps.

Step $1)$ Prove $WH_{\rm at}^{\V,q,s}(\Omega)\subseteq\WH$. To prove this,
let $f\in WH_{\rm at}^{\V,q,s}(\Omega)$. Then, by Definition \ref{def-ats},
we know that there exists a sequence of $(\varphi,q)_s$-atoms, $\{a^k\}_{k\in\zz}$,
related to stopping times $\{\nu^k\}_{k\in\zz}$,
such that, for any $n\in\mathbb{Z}_{+}$,
$$f_n=\sum_{k\in\mathbb{Z}}\mu^ka^k_n \quad {\rm \MP-almost\ everywhere},$$
where $\mu^k:=\widetilde{C}2^k\|\mathbf{1}_{B_{\nu^k}}\|_{L^{\varphi}(\Omega)}$ for any $k\in\zz$ and
$\widetilde{C}$ is a positive constant independent of $f$.
By the definitions of $\WH$ and $WH_{\rm at}^{\V,q,s}(\Omega)$, it suffices to prove that,
for any $\alpha,$ $\lambda\in\D,$
\begin{align}\label{ats}
\V\lf(\lf\{x\in\Omega:\ s(f)(x)>\alpha\r\},\frac{\alpha}{\lambda}\r)
\lesssim\sup_{k\in\zz}\V\lf(B_{\nu^k},\frac{2^k}{\lambda}\r).
\end{align}
To this end, for any fixed $\alpha\in\D,$ let $k_0\in\zz$ be such that $2^{k_0}\le\alpha<2^{k_0+1}.$
Combining this and the subadditivity of operator $s$, we conclude that, for any $\lambda\in\D,$
\begin{align*}
\V\lf(\lf\{x\in\Omega:\ s(f)(x)>\alpha\r\},\frac{\alpha}{\lambda}\r)
&\le\V\lf(\lf\{x\in\Omega:\ \sum_{k\in\zz}\mu^ks\lf(a^k\r)(x)
>\alpha\r\},\frac{\alpha}{\lambda}\r)\\
&\le\V\lf(\lf\{x\in\Omega:\ \sum_{k=-\fz}^{k_0-1}\mu^ks\lf(a^k\r)(x)
>2^{k_0-1}\r\},\frac{2^{k_0+1}}{\lambda}\r)\\
&\hs+\V\lf(\lf\{x\in\Omega:\ \sum_{k=k_0}^{\fz}\mu^ks\lf(a^k\r)(x)
>2^{k_0-1}\r\},\frac{2^{k_0+1}}{\lambda}\r)
=:{\rm I_{\alpha,1}+I_{\alpha,2}}.
\end{align*}
Thus, in order to show \eqref{ats}, we only need to estimate ${\rm I_{\alpha,1}}$
and ${\rm I_{\alpha,2}}$.

We first estimate ${\rm I_{\alpha,1}}$. For any $r\in(\max\{p_{\V}^{+},1\},\fz)$ and
$\ell\in(0,1-\frac{\max\{p_{\V}^{+},1\}}{r})$, by
the H\"older inequality, the monotone convergence theorem and the definition of
$L_{\V}^{r}(\Omega)$, we know that, for any $\lambda\in\D$,
\begin{align}\label{I1}
{\rm I_{\alpha,1}}
&\le\frac{1}{2^{(k_0-1){r}}}\dis\int_{\Omega}
\lf[\sum_{k=-\fz}^{k_0-1}\mu^ks\lf(a^k\r)(x)\r]^{r}
\V\lf(x,\frac{2^{k_0+1}}{\lambda}\r)\,d\MP\\
&\le\frac{1}{2^{(k_0-1)r}}\dis\int_{\Omega}
\lf(\sum_{k=-\fz}^{k_0-1}2^{k\ell r'}\r)^{\frac{{r}}{r'}}
\lf\{\sum_{k=-\fz}^{k_0-1}2^{-k\ell{r}}\lf[\mu^ks\lf(a^k\r)(x)\r]^{r}\r\}
\V\lf(x,\frac{2^{k_0+1}}{\lambda}\r)\,d\MP\noz\\
&\le2^{-{r}(k_0-1)(1-\ell)}\lf(1-2^{-\ell r'}\r)^{-r/r'}
\sum_{k=-\fz}^{k_0-1}2^{-k\ell{r}}\lf(\mu^k\r)^{r}
\lf\|s\lf(a^k\r)\r\|_{L_{\V}^{r}(\Omega)}^{r}
\int_{B_{\nu^k}}\V\lf(x,\frac{2^{k_0+1}}{\lambda}\r)\,d\MP\noz.
\end{align}
For the case $q\in(\max\{p_{\V}^{+},1\},\fz)$, let $r:=q$. From \eqref{I1},
the uniformly upper type $p_{\V}^{+}$ property of $\V$ and
the fact that $a^k$ is a $(\varphi,q)_s$-atom for any $k\in\zz$, we deduce that,
for any $\lambda\in\D$,
\begin{align*}
{\rm I_{\alpha,1}}
&\le2^{-{q}(k_0-1)(1-\ell)}\lf(1-2^{-\ell q'}\r)^{-q/q'}
\sum_{k=-\fz}^{k_0-1}2^{-k\ell{q}}\lf(\mu^k\r)^{q}
\lf\|s\lf(a^k\r)\r\|_{L_{\V}^{q}(\Omega)}^{q}
\int_{B_{\nu^k}}\V\lf(x,\frac{2^{k_0+1}}{\lambda}\r)\,d\MP\\
&\le 2^{-{q}(k_0-1)(1-\ell)}\lf(1-2^{-\ell q'}\r)^{-q/q'}
\sum_{k=-\fz}^{k_0-1}2^{-k\ell q}\lf(\widetilde{C}2^k\r)^{q}2^{(k_0+1-k)p_{\V}^{+}}
\sup_{k\in\zz}\V\lf(B_{\nu^k},\frac{2^k}{\lambda}\r),
\end{align*}
which, together with $(1-\ell)q>p_{\V}^{+}$, implies that, for any $\lambda\in\D$,
\begin{align}\label{I10}
{\rm I_{\alpha,1}}
&\le \lf(\widetilde{C}\r)^{q}
2^{-{q}(k_0-1)(1-\ell)}\lf(1-2^{-\ell q'}\r)^{-q/q'}2^{(k_0+1)p_{\V}^{+}}
\sum_{k=-\fz}^{k_0-1}2^{k[(1-\ell)q-p_{\V}^{+}]}
\sup_{k\in\zz}\V\lf(B_{\nu^k},\frac{2^k}{\lambda}\r)\\
&\le \lf(\widetilde{C}\r)^{q}\lf(1-2^{-\ell q'}\r)^{-q/q'}
\lf[1-2^{p_{\V}^{+}-(1-\ell)q}\r]\sup_{k\in\zz}\V\lf(B_{\nu^k},\frac{2^k}{\lambda}\r)\noz.
\end{align}
Letting $\ell\to0^{+}$ in \eqref{I10}, we conclude that, for any given
$q\in(\max\{p_{\V}^{+},1\},\fz)$ and any $\lambda\in\D$,
\begin{align}\label{I11}
{\rm I_{\alpha,1}}\lesssim \lf(1-2^{p_{\V}^{+}-q}\r)  \sup_{k\in\zz}\V\lf(B_{\nu^k},\frac{2^k}{\lambda}\r)\lesssim \sup_{k\in\zz}\V\lf(B_{\nu^k},\frac{2^k}{\lambda}\r).
\end{align}
For the case $q=\fz$, notice that, for any $k\in\zz$,
$\|s(a^k)\|_{L_{\V}^{r}(\Omega)}\le\|s(a^k)\|_{L^{\fz}(\Omega)}$.
Combining this and \eqref{I1}, similarly to the estimation of \eqref{I10},
we know that, for any $r\in(\max\{p_{\V}^{+},1\},\fz)$,
$\ell\in(0,1-\frac{\max\{p_{\V}^{+},1\}}{r})$ and $\lambda\in\D$,
\begin{align*}
{\rm I_{\alpha,1}}
&\lesssim \lf(1-2^{-\ell r'}\r)^{-r/r'}
\lf[1-2^{p_{\V}^{+}-(1-\ell)r}\r]\sup_{k\in\zz}\V\lf(B_{\nu^k},\frac{2^k}{\lambda}\r).
\end{align*}
Letting $r:=p_{\V}^{+}+1$ and $\ell\to0^{+}$ in the above inequality,
we finally find that, for any $\lambda\in\D$,
\begin{align}\label{I12}
{\rm I_{\alpha,1}}\lesssim \sup_{k\in\zz}\V\lf(B_{\nu^k},\frac{2^k}{\lambda}\r).
\end{align}

Now we estimate ${\rm I_{\alpha,2}}$. For any $k\in\zz$, by the definition of $a^k$,
we have
$$\lf\{x\in\Omega:\ s\lf(a^k\r)(x)\not=0\r\}\subseteq B_{v^k}.$$
From this, it follows that
$$0\le s\lf(f_{\alpha,2}\r)\le\sum_{k=k_0}^{\fz}\mu^ks\lf(a^k\r)
=\sum_{k=k_0}^{\fz}\mu^ks\lf(a^k\r)\mathbf{1}_{B_{\nu^k}},$$
which implies that
$$\lf\{x\in\Omega:\ s\lf(f_{\alpha,2}\r)(x)\not=0\r\}\subseteq \bigcup_{k=k_0}^{\fz}B_{v^k}.$$
Combining this, the fact that $\V$ is of uniformly lower type $p_{\V}^{-}$
and of uniformly upper type $p_{\V}^{+}$, we obtain, for any $\lambda\in\D$,
\begin{align}\label{I2}
{\rm I_{\alpha,2}}
&\le\sum_{k=k_0}^{\fz}\V\lf(B_{\nu^k},\frac{2^{k_0+1}}{\lambda}\r)
\lesssim \sum_{k=k_0}^{\fz}2^{p_{\V}^{+}}\V\lf(B_{\nu^k},\frac{2^{k_0}}{\lambda}\r)\\
&\lesssim2^{p_{\V}^{+}}\sum_{k=k_0}^{\fz}2^{(k_0-k)p_{\V}^{-}}
\V\lf(B_{\nu^k},\frac{2^{k}}{\lambda}\r)
\lesssim\sup_{k\in\zz}\V\lf(B_{\nu^k},\frac{2^k}{\lambda}\r).\noz
\end{align}
From \eqref{I11}, \eqref{I12} and \eqref{I2}, it follow that,
for any $\alpha,$ $\lambda\in\D,$
\eqref{ats} holds true, which further implies that
$\|f\|_{\WH}\lesssim\|f\|_{WH_{\rm at}^{\V,q,s}(\Omega)}$.
Thus, we have $WH_{\rm at}^{\V,q,s}(\Omega)\subseteq\WH$. This finishes the proof of Step $1)$.

Step $2)$ Prove $\WH\subseteq WH_{\rm at}^{\V,q,s}(\Omega).$
To this end, let $f\in\WH$. For any $k\in\zz$ and $x\in\Omega$, let
$$\nu^k(x):=\inf\{n\in\nn:\ s_{n+1}(f)(x)>2^k\}\quad \mbox{and} \quad
\mu^k:=2^{k+1}\lf\|\mathbf{1}_{B_{\nu^k}}\r\|_{L^{\V}(\Omega)}.$$
Then $(\nu^k)_{k\in\zz}$ is a sequence of non-decreasing stopping times.
Moreover, for any $k\in\zz$ and $n\in\mathbb{Z}_{+}$,
if $\mu^k\neq0$, let $$a_n^k:=\frac{f_n^{\nu^{k+1}}-f_n^{\nu^{k}}}{\mu^k};$$
otherwise, let $a_n^k:=0$. Then we have
$$f_n=\sum_{k\in\zz}\mu^ka_n^k \quad \MP\mbox{-almost everywhere.}$$
Now we claim that, for any $k\in\zz$, $a^k:=(a_n^k)_{n\in\mathbb{Z}_{+}}$
is a $(\varphi,q)_s$-atom. Indeed, for any $k\in\zz$, it is clear that
$a^k$ is a martingale. When $\nu^k\geq n$, we easy know $a_n^k=0$. Thus,
$a^k$ satisfies Definition \ref{def-atom}(i). Similarly to the proof of
\cite[Theorem 1.4]{xjy18}, we know that, for any $k\in\zz$.
$$\lf\|s(a^k)\r\|_{L^{\fz}(\Omega)}\le \lf\|\mathbf{1}_{B_{\nu^k}}\r\|_{L^{\V}(\Omega)}^{-1}.$$
This implies that $a^k$ is an $L^2(\Omega)$-bounded martingale
and hence $(a_n^k)_{n\in\mathbb{Z}_{+}}$ converges in $L^2(\Omega)$ as $n\to\fz$.
Denoting this limit still by $a^k$,
then $\ee_n(a^k)=a_n^k$ for any $n\in\mathbb{Z}_{+}$.
Moreover, for any given $q\in(0,\fz]$ and any $k\in\zz$,
$$\left\|s(a^k)\right\|_{L_{\varphi}^q(B_{\nu^k})}
\le \|s(a^k)\|_{L^{\fz}(\Omega)}
\le \lf\|\mathbf{1}_{B_{\nu^k}}\r\|_{L^{\V}(\Omega)}^{-1}.$$
Thus, $a^k$ satisfies Definition \ref{def-atom}(ii) and hence $a^k$ is a $(\varphi,q)_s$-atom.
This proves the above claim. On another hand, for any $k\in\zz$, we have
$\{x\in\Omega:\ s(f)(x)>2^k\}=B_{\nu^k}$. From this, it follows that,
for any $\lambda\in\D$,
$$\sup_{k\in\zz}\V\lf(B_{\nu^k},\frac{2^k}{\lambda}\r)
=\sup_{k\in\zz}\V\lf(\{x\in\Omega:\ s(f)(x)>2^k\},\frac{2^k}{\lambda}\r)
\le \sup_{\alpha\in\D}\V\lf(\{x\in\Omega:\ s(f)(x)>\alpha\},\frac{\alpha}{\lambda}\r),$$
which implies that $f\in WH_{\rm at}^{\V,q,s}(\Omega)$ and
$\|f\|_{WH_{\rm at}^{\V,q,s}(\Omega)}\le \|f\|_{\WH}$. This finishes the proof of Step $2)$
and hence of Theorem \ref{thm-atom}.
\end{proof}

\begin{theorem}\label{thm-apq}
Let $\V$ be a Musielak--Orlicz function satisfying Assumption \ref{a1-new}.
Then
$$WP_{\V}(\Omega)=WH_{\rm at}^{\V,\fz,M}(\Omega)\quad \mbox{and} \quad
WQ_{\V}(\Omega)=WH_{\rm at}^{\V,\fz,S}(\Omega)\quad \mbox{with equivalent quasi-norms}.$$
\end{theorem}

\begin{proof}
The proof of this theorem is just a slight modification on that of Theorem \ref{thm-atom}.
For the convenience of the reader, we present some details.
We only give the proof for $WP_{\V}(\Omega)$
because the proof for $WQ_{\V}(\Omega)$ is similar.

We first prove $WP_{\V}(\Omega)\subseteq WH_{\rm at}^{\V,\fz,M}(\Omega)$.
To this end, let $f\in WP_{\V}(\Omega)$. For any $k\in\zz$, $n\in\nn$ and $x\in\Omega$, let
$$\nu^k(x):=\lf\{n\in\zz_{+}:\ \lambda_n(x)>2^k\r\},\quad
\mu^k:=3\cdot2^{k}\lf\|\mathbf{1}_{B_{\nu^k}}\r\|_{L^{\varphi}(\Omega)},
\quad \mbox{and}\quad a_n^k:=\frac{f_n^{\nu^{k+1}}-f_n^{\nu^k}}{\mu^k}$$
if $\mu^k\not=0$, otherwise, let $a_n^k:=0$, where
$(\lambda_n)_{n\in\zz_{+}}\in\Lambda[WP_{\varphi}](f)$.
Then, using the same method as that used in the proof of Theorem \ref{thm-atom},
we can prove that, for any $k\in\zz$,
$\|M(a^k)\|_{L^{\fz}(\Omega)}\le \|\mathbf{1}_{B_{\nu^k}}\|_{L^{\varphi}(\Omega)}^{-1}$,
$a^k$ is a $(\varphi,\fz)_M$-atom and
$\|f\|_{WH_{\rm at}^{\V,\fz,M}(\Omega)}\le \|f\|_{WP_{\V}(\Omega)}.$

Conversely, let $f\in WH_{\rm at}^{\V,\fz,M}(\Omega)$. Then there exist a
sequence of $(\varphi,\fz)_M$-atoms, $\{a^k\}_{k\in\zz}$,
related to stopping times $\{\nu^k\}_{k\in\zz}$ and a positive constant $\widetilde{C}$,
independent of $f$, such that, for any $n\in\mathbb{Z}_{+}$,
\begin{align*}
f_n=\sum_{k\in\mathbb{Z}}\widetilde{C}2^k
\|\mathbf{1}_{B_{\nu^k}}\|_{L^{\varphi}(\Omega)}a_n^k
\quad {\rm \MP-almost\ everywhere}.
\end{align*}
For any $n\in\mathbb{Z}_{+}$, let
$\lambda_n:=\sum_{k\in\mathbb{Z}}\widetilde{C}2^k\mathbf{1}_{\{x\in\Omega:\ \nu^k(x)\le n\}}$.
Then, by the definition of $a^k$, we know that
$(\lambda_n)_{n\in\mathbb{Z}_{+}}$
is a nonnegative adapted sequence and, for any $n\in\nn$,
$$|f_n|\le \sum_{k\in\mathbb{Z}}\widetilde{C}
2^k\|\mathbf{1}_{B_{\nu^k}}\|_{L^{\varphi}(\Omega)}
\|a_n^k\|_{L^{\fz}(\Omega)}\mathbf{1}_{\{x\in\Omega:\ \nu^k(x)\le n-1\}}\le \lambda_{n-1}
\quad {\rm \MP-almost\ everywhere}.$$
Now we show $\|\lambda_{\fz}\|_{\WP}\lesssim \|f\|_{WH_{\rm at}^{\V,\fz,M}(\Omega)}$.
For any fixed $\alpha\in\D$, let $k_0\in\zz$ be
such that $2^{k_0}\le\alpha<2^{k_0+1}$.
Similarly to the estimations of \eqref{I12} and \eqref{I2}
via replacing $\mu^ks(a^k)$ by $\widetilde{C}2^k\mathbf{1}_{B_{\nu^k}}$,
we conclude that, for any $\gamma\in\D$,
\begin{align*}
\V\lf(\{x\in\Omega:\ \lambda_{\fz}(x)>\alpha\},\frac{\alpha}{\gamma}\r)
&\le\V\lf(\lf\{x\in\Omega:\ \sum_{k=-\fz}^{k_0-1}\widetilde{C}2^k\mathbf{1}_{B_{\nu^k}}
>2^{k_0-1}\r\},\frac{2^{k_0+1}}{\gamma}\r)\\
&\hs+\V\lf(\lf\{x\in\Omega:\ \sum_{k=k_0}^{\fz}\widetilde{C}2^k\mathbf{1}_{B_{\nu^k}}
>2^{k_0-1}\r\},\frac{2^{k_0+1}}{\gamma}\r)
\lesssim\sup_{k\in\zz}\V\lf(B_{\nu^k},\frac{2^k}{\gamma}\r).
\end{align*}
This implies that $f\in WP_{\V}(\Omega)$ and
$\|f\|_{WP_{\V}(\Omega)}\le \|\lambda_{\fz}\|_{\WP}
\lesssim \|f\|_{WH_{\rm at}^{\V,\fz,M}(\Omega)},$
which completes the proof of Step $2)$ and hence of Theorem \ref{thm-apq}.
\end{proof}

\begin{remark}\label{rmk-cac}
\begin{enumerate}
\item[{\rm (i)}]
For any given $p\in\D$, when $\V(x,t):=t^p$ for any $x\in\Omega$ and $t\in\D$,
in this case, Theorem \ref{thm-atom} with $q=\fz$ for Vilenkin martingales was
investigated by Weisz \cite[Theorem 1]{w98}, and then Theorem \ref{thm-atom} with $q=\fz$
and Theorem \ref{thm-apq} were obtained by Hou and Ren \cite[Theorems 1, 2 and 3]{hr06}.
Observing that Theorem \ref{thm-atom} includes the $q$-atomic characterization
of $\WH$ for any $q\in(\max\{p,1\},\fz)$,
Theorem \ref{thm-atom} generalizes and improves \cite[Theorem 1]{w98}
and \cite[Theorem 1]{hr06}. Moreover, since $\V$ is of wide generality,
we know that Theorem \ref{thm-apq} generalizes \cite[Theorems 2 and 3]{hr06}.

\item[{\rm (ii)}] Let $\Phi$ be an Orlicz function. Theorem \ref{thm-atom}
with $q=\fz$ and Theorem \ref{thm-apq} when $\V(x,t):=\Phi(t)$ for any $x\in\Omega$
and $t\in\D$ were obtained by Jiao et al. \cite[Theorems 2.1 and 2.4]{jwp15}
under some slightly stronger assumptions.
Indeed, \cite[Theorems 2.1 and 2.4]{jwp15} require that
$\Phi\in \mathcal{G}_{\ell}$ for some $\ell\in(0,1]$ and the upper index
$q_{\Phi^{-1}}\in(0,\fz)$ [see \eqref{j} for its definition], however,
Theorems \ref{thm-atom} and \ref{thm-apq} in this case only need
$\Phi\in \mathcal{G}_{\ell}$ for some $\ell\in\D$. Moreover, in
Theorem \ref{thm-atom}, $q\in(\max\{p_{\Phi}^{+},1\},\fz]$ is much more than
the endpoint case $q=\fz$. Therefore, Theorems \ref{thm-atom}
and \ref{thm-apq} generalize and improve \cite[Theorems 2.1 and 2.4]{jwp15}.

\item[{\rm (iii)}] Theorem \ref{thm-atom} with $q=\fz$ and Theorem \ref{thm-apq}
were first proved by Yang \cite[Theorems 3.1, 3.2 and 3.3]{Y17} under
Assumption \ref{a1}. However, Theorems \ref{thm-atom}
and \ref{thm-apq} only need Assumption \ref{a1-new} which is much weaker than Assumption \ref{a1}.
Thus, Theorems \ref{thm-atom} and \ref{thm-apq} indeed
improve \cite[Theorems 3.1, 3.2 and 3.3]{Y17}.

\item[{\rm (iv)}] Let $p(\cdot)$ be a measurable function on $\Omega$ satisfying
$$0<p^{-}:=\inf_{x\in\Omega}p(x)\le p^{+}:=\sup_{x\in\Omega}p(x)<\fz.$$ Let
$\V(x,t):=t^{p(x)}$ for any $x\in\Omega$ and $t\in\D$.
Observe that, in this case, $\V$ is of uniformly lower type $p^{-}$
and of uniformly upper type
$p^{+}$. From this and Remark \ref{rmk-sc}, it follows that Theorems \ref{thm-atom}
and \ref{thm-apq} give the atomic characterizations of weak variable
martingale Hardy spaces, which are also new.
\end{enumerate}
\end{remark}

Now we establish the atomic characterizations
of $WH_{\V}^S(\Omega)$ and $WH_{\V}^M(\Omega)$.
To this end, we need an additional notion.
The stochastic basis $\{\mathcal{F}_n\}_{n\in\mathbb{Z}_{+}}$ is said
to be \emph{regular} if there exists a positive constant $R$ such that, for any $n\in\nn$,
\begin{align}\label{R}
f_n\le Rf_{n-1}
\end{align}
holds true for any nonnegative martingale $(f_n)_{n\in\mathbb{Z}_{+}}.$

The following technical lemma was proved in \cite[Lemma 4.7]{xwyj}.

\begin{lemma}\label{st1}
Let $w:=(w_n)_{n\in\zz_{+}}\in \mathbb{S}^{-}$ be a special weight.
If the stochastic basis $\{\mathcal{F}_n\}_{n\in\mathbb{Z}_{+}}$ is regular,
then, for any nonnegative adapted process $\gamma=(\gamma_n)_{n\in\mathbb{Z}_{+}}$
and any $\lambda\in(\|\gamma_0\|_{L^{\fz}(\Omega)},\fz)$, there exists
a stopping time $\tau_{\lambda}$ such that, for any $n\in\mathbb{Z}_{+}$,
\begin{align*}
\sup_{n\le\tau_{\lambda}(x)}\gamma_n(x)=:M_{\tau_{\lambda}}\gamma(x)\le\lambda,
\quad\quad\forall\,x\in\Omega,
\end{align*}
\begin{align*}
\lf\{x\in\Omega:\ M\gamma(x)>\lambda\r\}\subseteq \lf\{x\in\Omega:\ \tau_{\lambda}(x)<\fz\r\}
\end{align*}
and \begin{align*}
w\lf(\lf\{x\in\Omega:\ \tau_{\lambda}(x)<\fz\r\}\r)\le KR
w\lf(\lf\{x\in\Omega:\ M\gamma(x)>\lambda\r\}\r),
\end{align*}
where $K$ and $R$ are the same as in \eqref{WS} and \eqref{R}, respectively.
Moreover, for any $\lambda_1$, $\lambda_2\in\D$ with $\lambda_1<\lambda_2$,
$\tau_{\lambda_1}\le \tau_{\lambda_2}$.
\end{lemma}

\begin{theorem}\label{thm-atomSM}
Let $q\in(0,\fz)$ and $\V\in\ss^{-}$ be a Musielak--Orlicz
function satisfying Assumption \ref{a1-new}.
If $q\in(\max\{p_{\V}^{+},1\},\fz]$ and
the stochastic basis $\{\mathcal{F}_n\}_{n\in\mathbb{Z}_{+}}$ is regular,
then
$$WH_{\V}^S(\Omega)=WH_{\rm at}^{\V,q,S}(\Omega)
\quad \mbox{and}\quad WH_{\V}^M(\Omega)=WH_{\rm at}^{\V,q,M}(\Omega)
\quad \mbox{ with equivalent quasi-norms.}$$
\end{theorem}
\begin{proof}
We only prove this theorem for $WH_{\V}^S(\Omega)$,
because the proof for $WH_{\V}^M(\Omega)$ only needs a slight modification.
We do this by two steps.

Step $1)$ Prove $WH_{\V}^S(\Omega)\subseteq WH_{\rm at}^{\V,q,S}(\Omega).$
To this end, let $f\in WH_{\V}^S(\Omega)$.
For any $k\in\zz$ and for the nonnegative adapted sequence $\{S_n(f)\}_{n\in\zz_{+}}$,
by Lemma \ref{st1}, we know that there exists a stopping time $\nu^k\in\ct$ such that
\begin{align*}
\lf\{x\in\Omega:\ S(f)(x)>2^k\r\}\subseteq \lf\{x\in\Omega:\ \nu^k(x)<\fz\r\},
\end{align*}
\begin{align}\label{st3}
S_{\nu^k}(f)(x)\le 2^k,
\quad\quad\forall\,x\in\Omega\end{align}
and \begin{align}\label{st4}
\V\lf(\lf\{x\in\Omega:\ \nu^k(x)<\fz\r\},t\r)\le KR
\V\lf(\lf\{x\in\Omega:\ S(f)(x)>2^k\r\},t\r),\quad\quad \forall\,t\in\D,
\end{align}
where $K$ and $R$ are the same as in \eqref{WS} and \eqref{R}, respectively.
Moreover, for any $k\in\zz$, $\nu^k\le\nu^{k+1}$ and $\nu^k\to\fz$ as $k\to\fz.$
For any $k\in\zz$ and $n\in\zz_{+}$, let
$$\mu^k:=2^{k+1}\lf\|\mathbf{1}_{B_{\nu^k}}\r\|_{L^{\varphi}(\Omega)},
\quad\mbox{and}\quad a_n^k:=\frac{f_n^{\nu^{k+1}}-f_n^{\nu^{k}}}{\mu^k}$$
if $\mu^k\not=0$, otherwise, let $a_n^k:=0$.
Then, for any $n\in\nn$, $f_n(x)=\sum_{k\in\zz}\mu^ka_n^k(x)$
for almost every $x\in\Omega$.
Now, we claim that, for any fixed $k\in\zz$, $a^k:=(a_n^k)_{n\in\zz_{+}}$
is a $(\varphi,q)_S$-atom. Indeed, it is clear that $(a_n^k)_{n\in\zz_{+}}$
is a martingale. Moreover, by \eqref{st3}, we know that
\begin{align}\label{atS2}
\lf[S(a^k)\r]^2&=\sum_{n\in\nn}\lf|d_na^k\r|^2
=\frac{1}{(\mu^k)^2}\sum_{n\in\nn}\lf|d_nf^{\nu^{k+1}}-d_nf^{\nu^{k}}\r|^2\\
&=\frac{1}{(\mu^k)^2}\sum_{n\in\nn}\lf|d_nf\mathbf{1}_{\{x\in\Omega:\
\nu^k(x)<n\le\nu^{k+1}(x)\}}\r|^2\le \frac{1}{(\mu^k)^2}\lf[S_{\nu^{k+1}}(f)\r]^2
\le\lf(\frac{2^{k+1}}{\mu^k}\r)^2.\noz
\end{align}
From this, it follows that $a^k$ is an $L^2(\Omega)$-bounded martingale and hence
$(a_n^k)_{n\in\zz_{+}}$ converges in $L^2(\Omega)$ as $n\to\fz$. Denoting its limit still by
$a^k$, then $\ee_n(a^k)=a_n^k.$
For any $n\in\zz_{+}$ and $x\in\{x\in\Omega:\ \nu^k(x)\geq n\}$,
by the definition of $f_n^{\nu^{k}}$, we know that $a_n^k(x)=0.$ Thus, $a^k$
satisfies Definition \eqref{def-atom}(i). From \eqref{atS2}, it follows that
$$\lf\|S(a^k)\r\|_{L_{\varphi}^q(B_{\nu^k})}
\le \lf\|S(a^k)\r\|_{L^{\fz}(\Omega)}
\le \|\mathbf{1}_{B_{\nu^k}}\|_{L^{\varphi}(\Omega)}^{-1},$$
which implies that $a^k$ satisfies Definition \eqref{def-atom}(ii)
and hence $a^k$ is a $(\varphi,q)_S$-atom. This proves the above claim.

Now, we show $f\in WH_{\rm at}^{\V,q,S}(\Omega)$.
From \eqref{st4}, we deduce that, for any $k\in\zz$ and $\lambda\in\D,$
$$\V\lf(B_{\nu^k},\frac{2^k}{\lambda}\r)\le RK
\V\lf(\lf\{x\in\Omega:\ S(f)(x)>2^k\r\},\frac{2^k}{\lambda}\r)
\le RK\sup_{\alpha\in\D}\V\lf(\lf\{x\in\Omega:\ S(f)(x)>\alpha\r\},\frac{\alpha}{\lambda}\r).$$
This implies that $\|f\|_{WH_{\rm at}^{\V,q,S}(\Omega)}\lesssim \|f\|_{WH_{\V}^S(\Omega)},$
which completes the proof of Step $1)$.

Step $2)$ Prove $WH_{\rm at}^{\V,q,S}(\Omega)\subseteq WH_{\V}^S(\Omega)$.
To prove this, let $f\in WH_{\rm at}^{\V,q,S}(\Omega)$. Then there exists a
sequence of triples, $\{\mu^k,a^k,\nu^k\}_{k\in\zz}$, such that
$f=\sum_{k\in\zz}\mu^ka^k$ pointwise, where $\{a^k\}_{k\in\zz}$ are $(\V,q)_S$-atoms,
$\{\nu^k\}_{k\in\zz}$ are the stopping times associated with $\{a^k\}_{k\in\zz}$,
$\mu^k:=\widetilde{C}2^k\|\mathbf{1}_{B_{\nu^k}}\|_{L^{\varphi}(\Omega)}$ for any $k\in\zz$
and $\widetilde{C}$ is a positive constant independent of $f$.

Now, we prove that $f\in WH_{\V}^S(\Omega).$ For any fixed $\alpha\in\D$,
let $k_0\in\zz$ be such that $2^{k_0}\le\alpha< 2^{k_0+1}$.
Then, by the arguments same as in the estimations of
\eqref{I11}, \eqref{I12} and \eqref{I2}, we find that, for any $\lambda\in\D,$
\begin{align*}
\V\lf(\lf\{x\in\Omega:\ S(f)(x)>\alpha\r\},\frac{\alpha}{\lambda}\r)
&\le\V\lf(\lf\{x\in\Omega:\ \sum_{k=-\fz}^{k_0-1}\mu^kS\lf(a^k\r)(x)
>2^{k_0-1}\r\},\frac{2^{k_0+1}}{\lambda}\r)\\
&\hs+\V\lf(\lf\{x\in\Omega:\ \sum_{k=k_0}^{\fz}\mu^kS\lf(a^k\r)(x)
>2^{k_0-1}\r\},\frac{2^{k_0+1}}{\lambda}\r)\\
&\lesssim\sup_{k\in\zz}\V\lf(B_{\nu^k},\frac{2^k}{\lambda}\r),
\end{align*}
which implies that
$\|f\|_{WH_{\V}^S(\Omega)}\lesssim \|f\|_{WH_{\rm at}^{\V,q,S}(\Omega)}$
and hence $f\in WH_{\V}^S(\Omega).$
This finishes the proof of Theorem \ref{thm-atomSM}.
\end{proof}

\begin{remark}\label{rmk-cc}
\begin{enumerate}
\item[{\rm (i)}] Let $\Phi$ be an Orlicz function. Theorem \ref{thm-atomSM} with $q=\fz$,
when $\V(x,t):=\Phi(t)$ for any $x\in\Omega$ and $t\in\D$, was proved
by Jiao et al. \cite[Theorem 2.3]{jwp15} under the regularity assumption and
the assumptions that $\Phi\in \mathcal{G}_{\ell}$ for some $\ell\in(0,1]$
and the upper index $q_{\Phi^{-1}}\in(0,\fz)$.
However, Theorem \ref{thm-atomSM}, in this case, only needs
$\Phi\in \mathcal{G}_{\ell}$ for some $\ell\in\D$
and the regularity condition. Moreover, Theorem \ref{thm-atomSM} includes the
$q$-atomic characterizations for any $q\in(\max\{p_{\V}^{+},1\},\fz)$.
Thus, Theorem \ref{thm-atomSM} generalize and improve \cite[Theorem 2.3]{jwp15}.

\item[{\rm (ii)}] Let $p\in\D$ and $w$ be a special weight. If $\V(x,t):=w(x)t^p$
for any $x\in\Omega$ and $t\in\D$, then Theorems \ref{thm-atom}, \ref{thm-apq}
and \ref{thm-atomSM} give the atomic characterizations of weak weighted martingale
Hardy spaces, which are also new.
\end{enumerate}
\end{remark}

\section{Boundedness of sublinear operators\label{s4}}

In this section, we first obtain the boundedness of sublinear
operators from $WH_{\varphi}^s(\Omega)$
[resp., $WH_{\varphi}^M(\Omega)$, $WH_{\varphi}^S(\Omega)$, $WP_{\varphi}(\Omega)$ or $WQ_{\varphi}(\Omega)$] to $\WP$, and then clarify relations
among these weak martingale Musielak--Orlicz Hardy spaces.

\begin{theorem}\label{thm-sub}
Let $\V$ be a Musielak--Orlicz function satisfying Assumption \ref{a1-new}
and $T$ a sublinear operator satisfying Assumption \ref{a2-new}(i).
If there exists a positive constant $C$ such that,
for any $(\varphi,\fz)_s$-atom $a$ and any $t\in\D$,
\begin{align}\label{sub1}
\V\lf(\lf\{x\in\Omega:\ |T(a)(x)|>0\r\},t\r)\le C\V\lf(B_{\nu},t\r),
\end{align}
where $\nu$ is the stopping time associated with $a$,
then there exists a positive constant $C$ such that, for any $f\in\WH$,
\begin{align}\label{sub20}
\|Tf\|_{\WP}\le C\|f\|_{\WH}.
\end{align}
\end{theorem}

\begin{proof}
Let $f\in\WH$. By Step 2) of the proof of Theorem \ref{thm-atom}, we know
that there exists a sequence of
$(\varphi,\fz)_s$-atoms $\{a^k\}_{k\in\zz}$, related to
stopping times $\{\nu^k\}_{k\in\zz}$,
such that, for any $\lambda\in\D$,
$$f=\sum_{k\in\zz}\mu^ka^k \quad \mbox{and} \quad
\sup_{k\in\zz}\V\lf(B_{\nu^k},\frac{2^k}{\lambda}\r)\le
\sup_{\alpha\in\D}\V\lf(\{x\in\Omega:\ s(f)(x)>\alpha\},\frac{\alpha}{\lambda}\r) $$ where
$\mu^k:=\widetilde{C}2^k\|\mathbf{1}_{B_{\nu^k}}\|_{L^{\varphi}(\Omega)}$
for any $k\in\zz$ and $\widetilde{C}$ is a positive constant
independent of $f$. Thus, in order to prove \eqref{sub20},
we only need to prove that, for any $\alpha,$ $\lambda\in\D,$
\begin{align}\label{sub2}
\V\lf(\lf\{x\in\Omega:\ |T(f)(x)|>\alpha\r\},\frac{\alpha}{\lambda}\r)
\lesssim \sup_{k\in\zz}\V\lf(B_{\nu^k},\frac{2^k}{\lambda}\r).
\end{align}
For any fixed $\alpha\in\D$, let $k_0\in\zz$ be such that $2^{k_0}\le\alpha<2^{k_0+1}$.
Then, from the definition of $T$, it follows that, for any $\lambda\in\D$,
\begin{align*}
\V\lf(\lf\{x\in\Omega:\ |T(f)(x)|>\alpha\r\},\frac{\alpha}{\lambda}\r)
&\lesssim \V\lf(\lf\{x\in\Omega:\ \sum_{k=-\fz}^{k_0-1}\mu^k\lf|T\lf(a^k\r)(x)\r|
>\frac{\alpha}{2}\r\},\frac{\alpha}{\lambda}\r) \\
&\hs+\V\lf(\lf\{x\in\Omega:\ \sum_{k=k_0}^{\fz}\mu^k\lf|T\lf(a^k\r)(x)\r|>
\frac{\alpha}{2}\r\},\frac{\alpha}{\lambda}\r)\\
&\lesssim \V\lf(\lf\{x\in\Omega:\ \sum_{k=-\fz}^{k_0-1}\mu^k\lf|T\lf(a^k\r)(x)\r|
>2^{k_0-1}\r\},\frac{2^{k_0+1}}{\lambda}\r) \\
&\hs+\V\lf(\lf\{x\in\Omega:\ \sum_{k=k_0}^{\fz}\mu^k\lf|T\lf(a^k\r)(x)\r|>
2^{k_0-1}\r\},\frac{2^{k_0+1}}{\lambda}\r)=:{\rm I_1}+{\rm I_2}.
\end{align*}
Thus, to show \eqref{sub2}, we only need to estimate
${\rm I_1}$ and ${\rm I_2}$, respectively.

To estimate ${\rm I_1}$, we consider two cases.

Case $1)$ $q\in(1,\fz)\cap (p_{\V}^{+},\fz)$. In this case,
for any $\ell\in(0,1-\frac{p_{\V}^{+}}{q})$, by the H\"older inequality and
the boundedness of $T$, we know that, for any $\lambda\in\D,$
\begin{align*}
{\rm I_1} &\lesssim\frac{1}{2^{(k_0-1)q}}\dis\int_{\Omega}\lf[\sum_{k=-\fz}^{k_0-1}\mu^k
\lf|T\lf(a^k\r)(x)\r|\r]^q\V\lf(x,\frac{2^{k_0+1}}{\lambda}\r)\,d\MP\\
&\lesssim\frac{1}{2^{(k_0-1)q}}\lf(\sum_{k=-\fz}^{k_0-1}2^{k\ell q'}\r)^{\frac{q}{q'}}
\dis\int_{\Omega}\sum_{k=-\fz}^{k_0-1}2^{-k\ell q}\lf(\mu^k\r)^q
\lf|T\lf(a^k\r)(x)\r|^q\V\lf(x,\frac{2^{k_0+1}}{\lambda}\r)\,d\MP\\
&\lesssim 2^{-{q}(k_0-1)(1-\ell)}\lf(1-2^{-\ell q'}\r)^{-q/q'}
\sum_{k=-\fz}^{k_0-1}2^{-k\ell q}\lf(\mu^k\r)^q
\dis\int_{\Omega}\lf|s(a^k)(x)\r|^q\V\lf(x,\frac{2^{k_0+1}}{\lambda}\r)\,d\MP.
\end{align*}
From this, $q(1-\ell)>p_{\V}^{+}$ and the fact that $a^k$ is a
$(\varphi,\fz)_s$-atom for any $k\in\zz$,
we deduce that, for any $\lambda\in\D,$
\begin{align*}
{\rm I_1}
&\lesssim 2^{-{q}(k_0-1)(1-\ell)}\lf(1-2^{-\ell q'}\r)^{-q/q'}
\sum_{k=-\fz}^{k_0-1}2^{-k\ell q}\lf(\mu^k\r)^q
\lf\|s(a^k)\r\|^q_{L^{\fz}(B_{\nu^k})}\V\lf(B_{\nu^k},\frac{2^{k_0+1}}{\lambda}\r)\\
&\lesssim 2^{-{q}(k_0-1)(1-\ell)}\lf(1-2^{-\ell q'}\r)^{-q/q'}
\sum_{k=-\fz}^{k_0-1}2^{-k\ell q}2^{kq}2^{(k_0+1-k)p_{\V}^{+}}
\V\lf(B_{\nu^k},\frac{2^{k}}{\lambda}\r)\\
&\lesssim \lf(1-2^{-\ell q'}\r)^{-q/q'}
\lf[1-2^{p_{\V}^{+}-(1-\ell)q}\r]\sup_{k\in\zz}\V\lf(B_{\nu^k},\frac{2^k}{\lambda}\r).
\end{align*}
Letting $\ell\to0$ in above inequality, we conclude that, for any $\lambda\in\D$,
\begin{align}\label{sub3}
{\rm I_1}\lesssim\sup_{k\in\zz}\V\lf(B_{\nu^k},\frac{2^k}{\lambda}\r).
\end{align}

Case $2)$ $q\in(0,1]\cap(p_{\V}^{+},\fz)$. From the boundedness of $T$
and the uniformly upper type $p_{\V}^{+}$ property of $\V$, it follows that,
for any $\lambda\in\D$,
\begin{align}\label{sub31}
{\rm I_1}
&\lesssim\frac{1}{2^{(k_0-1)q}}\dis\int_{\Omega}\lf[\sum_{k=-\fz}^{k_0-1}\mu^k
\lf|T\lf(a^k\r)(x)\r|\r]^q\V\lf(x,\frac{2^{k_0+1}}{\lambda}\r)\,d\MP\\
&\lesssim\frac{1}{2^{(k_0-1)q}}\sum_{k=-\fz}^{k_0-1}\lf(\mu^k\r)^q\dis\int_{\Omega}
\lf|T\lf(a^k\r)(x)\r|^q\V\lf(x,\frac{2^{k_0+1}}{\lambda}\r)\,d\MP\noz\\
&\lesssim \frac{1}{2^{(k_0-1)q}}\sum_{k=-\fz}^{k_0-1}\lf(\mu^k\r)^q\lf\|s(a^k)\r\|^q_{L^{\fz}(\Omega)}
\V\lf(B_{\nu^k},\frac{2^{k_0+1}}{\lambda}\r)\noz\\
&\lesssim \frac{1}{2^{(k_0-1)q}}\sum_{k=-\fz}^{k_0-1}2^{kq}2^{(k_0+1-k)p_{\V}^{+}}
\sup_{k\in\zz}\V\lf(B_{\nu^k},\frac{2^{k}}{\lambda}\r)
\sim \sup_{k\in\zz}\V\lf(B_{\nu^k},\frac{2^k}{\lambda}\r).\noz
\end{align}

Now we estimate ${\rm I_2}.$ Clearly,
$$\lf\{x\in\Omega:\ \sum_{k=k_0}^{\fz}\mu^k\lf|T\lf(a^k\r)(x)\r|>2^{k_0-1}\r\}
\subseteq\bigcup_{k=k_0}^{\fz}\lf\{x\in\Omega:\ \lf|T\lf(a^k\r)(x)\r|>0\r\}.$$
Combining this, \eqref{sub1} and the fact that $\V$ is of uniformly lower type $p_{\V}^{-}$
and of uniformly upper type $p_{\V}^{+}$,
we find that, for any $\lambda\in\D,$
\begin{align*}
{\rm I_2}
&\lesssim \sum_{k=k_0}^{\fz}\V\lf(\lf\{x\in\Omega:\ \lf|T\lf(a^k\r)(x)\r|>0\r\},
\frac{2^{k_0+1}}{\lambda}\r)\\
&\lesssim 2^{p_{\V}^{+}}\sum_{k=k_0}^{\fz}\V\lf(B_{\nu^k},\frac{2^{k_0}}{\lambda}\r)
\lesssim \sum_{k=k_0}^{\fz}2^{(k_0-k)p_{\V}^{-}}\V\lf(B_{\nu^k},\frac{2^{k}}{\lambda}\r)
\lesssim \sup_{k\in\zz}\V\lf(B_{\nu^k},\frac{2^k}{\lambda}\r),
\end{align*}
which, together with \eqref{sub3} and \eqref{sub31},
further implies that \eqref{sub2} holds true.
This finishes the proof of Theorem \ref{thm-sub}.
\end{proof}

Using Theorems \ref{thm-apq} and \ref{thm-atomSM}, we can also
show that the sublinear operator $T$ is bounded from $WP_{\V}(\Omega)$
[resp., $WQ_{\V}(\Omega)$, $WH_{\V}^S(\Omega)$ or
$WH_{\V}^M(\Omega)$] to $\WP$, whose proofs are
similar to that of Theorem \ref{thm-sub}, the details being omitted.

\begin{theorem}\label{thm-subpq}
Let $\V$ be a Musielak--Orlicz function satisfying Assumption \ref{a1-new} and
$T$ a sublinear operator satisfying Assumption \ref{a2-new}(ii)
(resp., Assumption \ref{a2-new}(iii)). If there exists a positive constant $C$ such that,
for any $(\varphi,\fz)_S$-atom (resp., $(\varphi,\fz)_M$-atom) $a$ and any $t\in\D$,
\begin{align}\label{subpq}
\V\lf(\lf\{x\in\Omega:\ |T(a)(x)|>0\r\},t\r)\le C\V\lf(B_{\nu},t\r),
\end{align}
where $\nu$ is the stopping time associated with $a$,
then there exists a positive constant $C$ such that,
for any $f\in WQ_{\V}(\Omega)$ [resp., $f\in WP_{\V}(\Omega)$],
$$
\|Tf\|_{\WP}\le C\|f\|_{WQ_{\V}(\Omega)}\quad \quad
\lf[\mbox{resp., }  \|Tf\|_{\WP}\le C\|f\|_{WP_{\V}(\Omega)} \r].
$$
\end{theorem}

\begin{theorem}\label{thm-subSM}
Let $\V\in\ss^{-}$ be a Musielak--Orlicz function satisfying
Assumption \ref{a1-new} and
$T$ a sublinear operator satisfying Assumption \ref{a2-new}(ii)
(resp., Assumption \ref{a2-new}(iii)). If the stochastic basis
$\{\mathcal{F}_n\}_{n\in\mathbb{Z}_{+}}$ is regular and
there exists a positive constant $C$ such that, for any $(\varphi,\fz)_S$-atom
[resp., $(\varphi,\fz)_M$-atom] $a$ and any $t\in\D$,
\begin{align*}
\V\lf(\lf\{x\in\Omega:\ |T(a)(x)|>0\r\},t\r)\le C\V\lf(B_{\nu},t\r),
\end{align*}
where $\nu$ is the stopping time associated with $a$,
then there exists a positive constant
$C$ such that, for any $f\in WH_{\V}^M(\Omega)$ [resp., $f\in WH_{\V}^S(\Omega)$],
$$
\|Tf\|_{\WP}\le C\|f\|_{WH_{\V}^M(\Omega)}
\quad\quad
\lf[\mbox{resp., }  \|Tf\|_{\WP}\le C\|f\|_{WH_{\V}^S(\Omega)} \r].
$$
\end{theorem}

\begin{remark}\label{rmk-su}
\begin{enumerate}
\item[{\rm(i)}] For any given $p\in\D$,
when $\V(x,t):=t^p$ for any $x\in\Omega$ and $t\in\D$, Theorem \ref{thm-sub} for
Vilenkin martingales was originally obtained by Weisz \cite[Theorem 2]{w98}. Then
Theorems \ref{thm-sub} and \ref{thm-subpq}, in this case, were proved by
Hou and Ren \cite[Theorems 4, 5 and 6]{hr06}. Observe that the assumptions of
Theorem \ref{thm-sub} in this case are weaker than that of
\cite[Theorem 4]{hr06}. Indeed, the assumptions of \cite[Theorem 4]{hr06}
require that $T$ is bounded on $L^q(\Omega)$ for some $q\in[1,2]\cap(p,\infty)$ and that
\eqref{sub1} holds true in this case. Therefore, to prove our claim, we only need to show that
the boundedness of $T$ on $L^q(\Omega)$ for some $q\in[1,2]\cap(p,\infty)$ implies
the boundedness of $T$ from $H_q^s(\Omega)$ to $L^q(\Omega)$.
This follows immediately from the well-known fact that the operator $s$ is bounded on
$L^q(\Omega)$ for any $q\in(0,2]$ (see \cite[Theorem 2.11(i)]{W94}).
Similarly, we can also deduce that the assumptions of Theorem \ref{thm-subpq} in this case
are weaker than that of \cite[Theorems 5 and 6]{hr06}.
Thus, Theorems \ref{thm-sub} and \ref{thm-subpq} generalize and
improve \cite[Theorems 4, 5 and 6]{hr06}, respectively.

\item[{\rm(ii)}] Let $\Phi$ be an Orlicz function. Theorems \ref{thm-sub}
and \ref{thm-subpq} when $\V(x,t):=\Phi(t)$ for any $x\in\Omega$ and
$t\in\D$ were proved by Jiao et al. \cite[Theorem 3.1 and Remark 3.2]{jwp15}
under the assumptions that $\Phi\in \mathcal{G}_{\ell}$ for some $\ell\in(0,1]$, $p_{\Phi^{-1}}\in(1,\fz)$ and $q_{\Phi^{-1}}\in\D$ [see \eqref{j} for the
definitions of $p_{\Phi^{-1}}$ and $q_{\Phi^{-1}}$].
However, on the assumption on $\Phi$, Theorems \ref{thm-sub} and \ref{thm-subpq}
in this case only need that $\Phi\in \mathcal{G}_{\ell}$ for some $\ell\in\D$.
Thus, in this sense, Theorems \ref{thm-sub} and \ref{thm-subpq} totally
improve and generalize \cite[Theorem 3.1 and Remark 3.2]{jwp15}, respectively.

\item[{\rm(iii)}] Replacing Assumption \ref{a2-new} by Assumption \ref{a2},
Yang \cite[Theorems 4.2, 4.3 and 4.4]{Y17} also proved Theorems \ref{thm-sub}
and \ref{thm-subpq}. Clearly, Assumption \ref{a2-new} is quite weaker than Assumption \ref{a2}.
Thus, Theorems \ref{thm-sub} and \ref{thm-subpq}
improve \cite[Theorems 4.2, 4.3 and 4.4]{Y17}, respectively.
In particular, Theorem \ref{thm-subSM} is new.

\item[{\rm(iv)}] Let $p\in\D$ and $w$ be a special weight. If,
for any $x\in\Omega$ and $t\in\D$,  $\V(x,t):=w(x)t^p$, then
Theorems \ref{thm-sub}, \ref{thm-subpq} and \ref{thm-subSM} give
the boundedness of sublinear operators from weak weighted martingale
Hardy spaces to weak weighted Lebesgue spaces, which are also new.
\end{enumerate}
\end{remark}

The following weighted martingale inequalities come
from Bonami and L\'epingle \cite[Theorem 1]{bl79}
and Long \cite[Remark 6.6.12, Theorems 6.6.11 and 6.6.12]{Long93}.
\begin{theorem}\label{lem-bl}
Let $w$ be a special weight.
\begin{enumerate}
\item[\rm{(i)}]
If $w\in \mathbb{A}_{\fz}(\Omega)\cap \ss(\Omega)$ and $p\in[1,\fz)$, then
there exists a positive constant $C$ such that, for any $f\in H_{p}^M(\Omega,w\,d\MP)$,
\begin{align}\label{bl1}
\frac{1}{C}\|f\|_{H_{p}^M(\Omega,w\,d\MP)}\le \|f\|_{H_{p}^S(\Omega,w\,d\MP)}
\le C\|f\|_{H_{p}^M(\Omega,w\,d\MP)}.
\end{align}
\item[\rm{(ii)}]
If $w\in \ss^{-}(\Omega)$ and $p\in[2,\fz)$, then
there exists a positive constant $C$ such that, for any $f\in H_{p}^S(\Omega,w\,d\MP)$,
\begin{align}\label{bl2}
\|f\|_{H_{p}^s(\Omega,w\,d\MP)}\le C\|f\|_{H_{p}^S(\Omega,w\,d\MP)}.
\end{align}
\item[\rm{(iii)}]
If $w\in \ss^{+}(\Omega)$ and $p\in(0,2]$, then
there exists a positive constant $C$ such that, for any $f\in H_{p}^s(\Omega,w\,d\MP)$,
\begin{align}\label{bl3}
\|f\|_{H_{p}^S(\Omega,w\,d\MP)}\le C\|f\|_{H_{p}^s(\Omega,w\,d\MP)}.
\end{align}
\item[\rm{(iv)}]
If $w\in \mathbb{A}_{\fz}(\Omega)\cap \ss(\Omega)$ and $p\in(0,2]$, then
there exists a positive constant $C$ such that, for any $f\in H_{p}^s(\Omega,w\,d\MP)$,
\begin{align}\label{bl4}
\|f\|_{H_{p}^M(\Omega,w\,d\MP)}\le C\|f\|_{H_{p}^s(\Omega,w\,d\MP)}.
\end{align}
\end{enumerate}
\end{theorem}

\begin{theorem}\label{thm-mi}
Let $\V$ be a Musielak--Orlicz function satisfying Assumption \ref{a1-new}.
\begin{enumerate}
\item[\rm{(i)}]
If $\V\in \ss^{+}(\Omega)$ and $p_{\V}^{+}\in(0,2)$, then
there exists a positive constant $C$ such that, for any $f\in WH_{\V}^s(\Omega)$,
\begin{align}\label{mi1}
\|f\|_{WH_{\V}^S(\Omega)}\le C\|f\|_{WH_{\V}^s(\Omega)}.
\end{align}

\item[\rm{(ii)}]
If $\V\in \mathbb{A}_{\fz}(\Omega)\cap \ss(\Omega)$ and $p_{\V}^{+}\in(0,2)$, then
there exists a positive constant $C$ such that, for any $f\in WH_{\V}^s(\Omega)$,
\begin{align}\label{mi2}
\|f\|_{WH_{\V}^M(\Omega)}\le C\|f\|_{WH_{\V}^s(\Omega)}.
\end{align}

\item[\rm{(iii)}]
There exists a positive constant $C$ such that, for any $f\in WP_{\V}(\Omega)$
[resp., $f\in WQ_{\V}(\Omega)$],
\begin{align}\label{mi3}
\|f\|_{WH_{\V}^M(\Omega)}\le C\|f\|_{WP_{\V}(\Omega)}
\quad \lf[\mbox{resp., }\|f\|_{WH_{\V}^S(\Omega)}\le C\|f\|_{WQ_{\V}(\Omega)}\r].
\end{align}

\item[\rm{(iv)}]
If $\V\in \mathbb{A}_{\fz}(\Omega)\cap \ss(\Omega)$, then
there exists a positive constant $C$ such that, for any $f\in WP_{\V}(\Omega)$ [resp., $f\in WQ_{\V}(\Omega)$],
\begin{align}\label{mi4}
\|f\|_{WH_{\V}^S(\Omega)}\le C\|f\|_{WP_{\V}(\Omega)}\quad \mbox{and}
\quad \|f\|_{WH_{\V}^s(\Omega)}\le C\|f\|_{WP_{\V}(\Omega)}
\end{align}
$$\lf[\mbox{resp., }\|f\|_{WH_{\V}^M(\Omega)}\le C\|f\|_{WQ_{\V}(\Omega)}\r].$$

\item[\rm{(v)}]
If $\V\in \ss^{-}(\Omega)$, then
there exists a positive constant $C$ such that, for any $f\in WQ_{\V}(\Omega)$,
\begin{align}\label{mi5}
\|f\|_{WH_{\V}^s(\Omega)}\le C\|f\|_{WQ_{\V}(\Omega)}.
\end{align}

\item[\rm{(vi)}]
If $\V\in \mathbb{A}_{\fz}(\Omega)\cap \ss(\Omega)$ and $p_{\V}^{+}\in(0,2)$, then
there exists a positive constant $C$ such that, for any $f\in WQ_{\V}(\Omega)$,
\begin{align}\label{mi6}
\frac1C\|f\|_{WQ_{\V}(\Omega)}\le \|f\|_{WP_{\V}(\Omega)}\le C\|f\|_{WQ_{\V}(\Omega)}.
\end{align}
\end{enumerate}
Moreover, if $\{\cf_n\}_{n\in\zz_{+}}$ is regular and $\V\in \mathbb{A}_{\fz}(\Omega)$, then
$$WH_{\V}^s(\Omega)=WH_{\V}^M(\Omega)=WH_{\V}^S(\Omega)=WP_{\V}(\Omega)=WQ_{\V}(\Omega).$$
\end{theorem}

\begin{proof}
In order to prove \eqref{mi1} and \eqref{mi2}, we use Theorem \ref{thm-sub}
with the operator $T:=S$ or $M$. From Definition \ref{def-atom}(i), it follows that,
for any $(\varphi,q)_s$-atom $a$,
$$0\le\mathbf{1}_{\{x\in\Omega:\ \nu(x)=\fz\}}\lf[S(a)\r]^2
=\mathbf{1}_{\{x\in\Omega:\ \nu(x)=\fz\}}
\sum_{n\in\nn}\lf|d_na\r|^2
\le \sum_{n\in\nn}\mathbf{1}_{\{x\in\Omega:\ \nu(x)\geq n\}}\lf|d_na\r|^2=0,$$
which implies that $\{x\in\Omega:\ S(f)(x)>0\}\subseteq B_{\nu}$ and hence
the operator $S$ satisfies \eqref{sub1}.
Clearly, the Doob maximal operator $M$ also satisfies \eqref{sub1}.
By this, \eqref{bl3}, \eqref{bl4} and Theorem \ref{thm-sub},
we obtain \eqref{mi1} and \eqref{mi2}.

Inequalities \eqref{mi3} follow immediately from the definitions of $WP_{\V}(\Omega)$
and $WQ_{\V}(\Omega)$.

To prove inequalities \eqref{mi4} and \eqref{mi5}, we apply Theorem \ref{thm-subpq},
respectively, to the operator $T=S$, $M$ or $s$.
Observe that operators $M$, $S$ and $s$ all satisfy the condition \eqref{subpq}.
From \eqref{bl1} and \eqref{bl2}, it follows that, for any $q\in[2,\fz)$ and $t\in\D$,
$$s:\ H_q^M(\Omega,\V(\cdot,t)\,d\MP)\to L^q(\Omega,\V(\cdot,t))$$
is bounded. Combining this, \eqref{bl1}, \eqref{bl2} and Theorem \ref{thm-subpq},
we obtain \eqref{mi4} and \eqref{mi5}.

To show inequalities \eqref{mi6}, let $f\in WQ_{\V}(\Omega)$. For any $\varepsilon\in\D$,
there exists an adapted process
$\{\lambda_n^{(1)}\}_{n\in\zz_{+}}\in\Lambda[WQ_{\varphi}](f)$
such that, for any $n\in\nn$,
$$S_n(f)\le \lambda_{n-1}^{(1)}\quad{\rm and}\quad
\lf\|\lambda_{\fz}^{(1)}\r\|_{\WP}\le \|f\|_{WQ_{\V}(\Omega)}+\varepsilon.$$
By this, we find that, for any $n\in\nn$,
$$|f_n|\le M_{n-1}(f)+\lf|d_nf\r|\le M_{n-1}(f)+S_n(f)\le M_{n-1}(f)+\lambda_{n-1}^{(1)}.$$
Combining this and \eqref{mi4}, we know that
$$\|f\|_{WP_{\V}(\Omega)}\lesssim \|f\|_{WH_{\V}^M(\Omega)}+\lf\|\lambda_{\fz}^{(1)}\r\|_{\WP}
\lesssim \|f\|_{WQ_{\V}(\Omega)}+\varepsilon,$$
which, together with letting $\varepsilon\to0$, implies that $\|f\|_{WP_{\V}(\Omega)}\lesssim \|f\|_{WQ_{\V}(\Omega)}$
and $f\in WP_{\V}(\Omega)$. Moreover, for any $\varepsilon\in\D$,
there exists an adapted process
$\{\lambda_n^{(2)}\}_{n\in\zz_{+}}\in\Lambda[WP_{\varphi}](f)$ such that,
for any $n\in\nn$,
$$|f_n|\le \lambda_{n-1}^{(2)}\quad{\rm and}\quad
\lf\|\lambda_{\fz}^{(2)}\r\|_{\WP}\lesssim \|f\|_{WP_{\V}(\Omega)}+\varepsilon,$$
which implies that, for any $n\in\nn$,
$$S_n(f)\le S_{n-1}(f)+\lf|d_nf\r|\le S_{n-1}(f)+2\lambda_{n-1}^{(2)}.$$
From this and \eqref{mi4}, it follows that
$$\|f\|_{WQ_{\V}(\Omega)}\lesssim \|f\|_{WH_{\V}^S(\Omega)}+\lf\|\lambda_{\fz}^{(1)}\r\|_{\WP}
\lesssim \|f\|_{WP_{\V}(\Omega)}+\varepsilon$$
and hence, by letting $\varepsilon\to0$, $\|f\|_{WQ_{\V}(\Omega)}\lesssim \|f\|_{WP_{\V}(\Omega)}$.
Thus, we conclude that inequalities \eqref{mi6} hold true.

Finally, assume that $\{\mathcal{F}_n\}_{n\in\mathbb{Z}_{+}}$ is regular.
From this and  $\V\in \mathbb{A}_{\fz}(\Omega)$, it follows that
$\V\in\ss$ (see \cite[Proposition 6.3.7]{Long93}).
Then, by Theorems \ref{thm-apq} and \ref{thm-atomSM}, we have
$$WQ_{\V}(\Omega)=WH_{\V}^S(\Omega) \quad {\rm and}
\quad WP_{\V}(\Omega)=WH_{\V}^M(\Omega).$$
Combining this and \eqref{mi4}, we know that
$$WQ_{\V}(\Omega)=WH_{\V}^S(\Omega)=WH_{\V}^M(\Omega)
= WP_{\V}(\Omega)\subseteq WH_{\V}^s(\Omega).$$
Thus, to complete the proof of this theorem, we only need to show that
$WH_{\V}^s(\Omega)\subseteq WH_{\V}^S(\Omega)$.
By the regularity and \cite[Lemma 2.18]{W94},
we have $|d_nf|^2\lesssim \ee_{n-1}(|d_nf|^2)$ for any $n\in\nn$.
From this, it follows that $S(f)\lesssim s(f)$ and hence
$\|f\|_{WH_{\V}^S(\Omega)}\lesssim \|f\|_{WH_{\V}^s(\Omega)}$.
Thus, $WH_{\V}^s(\Omega)\subseteq WH_{\V}^S(\Omega)$,
which completes the proof of Theorem \ref{thm-mi}.
\end{proof}

\begin{remark}\label{rmk=mi}
\begin{enumerate}
\item[{\rm(i)}]
For any given $p\in\D$, if $\V(x,t):=t^p$ for any $x\in\Omega$
and $t\in\D$, then Theorem \ref{thm-mi} in this case coincides with \cite[Theorem 7]{hr06}.

\item[{\rm(ii)}] Let $\Phi$ be an Orlicz function. Theorem \ref{thm-mi}
when $\V(x,t):=\Phi(t)$ for any $x\in\Omega$ and $t\in\D$ was proved
by Jiao et al. \cite[Theorem 3.3]{jwp15} under some slightly stronger assumptions.
Indeed, \cite[Theorem 3.3]{jwp15} needs the condition that $\Phi$ is of lower type $p_{\Phi}^{-}$ for some $p_{\Phi}^{-}\in(0,1]$ and of upper type $p_{\Phi}^{+}:=1$
and $q_{\Phi^{-1}}\in\D$. However, the conclusions \eqref{mi1}, \eqref{mi2} and \eqref{mi6}
of Theorem \ref{thm-mi} only need $p_{\Phi}^{+}\in(0,2)$.
Thus, Theorem \ref{thm-mi} generalizes and improves \cite[Theorem 3.3]{jwp15}.

\item[{\rm(iii)}] Under Assumption \ref{a2}, Yang \cite[Therem 4.5]{Y17} also proved
the martingale inequalities among spaces $WH_{\varphi}^M(\Omega)$, $WH_{\varphi}^S(\Omega)$,
$WH_{\varphi}^s(\Omega)$, $WP_{\varphi}(\Omega)$ and $WQ_{\varphi}(\Omega)$.
By Assumption \ref{a2} and Remark \ref{rem-p}, we know that
Theorem \ref{thm-mi} improves \cite[Therem 4.5]{Y17}.

\item[{\rm(iv)}] Similarly to the discussion of Remark \ref{rmk-su}(iv),
Theorem \ref{thm-mi} is also new on weak weighted martingale (Orlicz) Hardy spaces.
\end{enumerate}
\end{remark}

\section{Convergence theorems\label{s5}}

In this section, we obtain bounded convergence theorems and
dominated convergence theorems on weak Musielak--Orlicz spaces $\WP$.
We begin with the following notion.

\begin{definition}
Let $\V$ be a Musielak--Orlicz function. The space $\WP$ is said to have
\emph{absolutely continuous quasi-norm} if, for any measurable function $f\in\WP$,
$$\lim_{n\to\fz}\|f\mathbf{1}_{\{x\in\Omega:\ |f(x)|>n\}}\|_{\WP}=0.$$
\end{definition}

But, not every weak Musielak--Orlicz spaces $\WP$ has
absolutely continuous quasi-norm even when $\V$ satisfies Assumption \ref{a1-new}.
For example, let $\Omega:=(0,1]$ and $\MP$ be the Lebesgue measure. For any
$x\in(0,1]$ and $t\in\D$, let $\V(x,t):=t^p$ with $p\in\D$ and
$f(x):=x^{-\frac1p}$ (see, for example, \cite[Example 2.5]{ly}).
Via a simple calculation, we know that $\|f\|_{\WP}=1$ and
$\V$ is of uniformly lower type $p$ and of uniformly upper type $p$.
However, for any $n\in\nn$, we have
$\|f\mathbf{1}_{\{x\in\Omega:\ |f(x)|>n\}}\|_{\WP}=1$.
Thus, $\WP$ for this $\V$ has no absolutely continuous quasi-norm.

\begin{definition}
Let $\V$ be a Musielak--Orlicz function.
The \emph{new Musielak--Orlicz space $W\mathcal{L}_{\V}(\Omega)$} is
defined as follows:
$$W\mathcal{L}_{\V}(\Omega):=
\lf\{f\in\WP:\ \lim_{n\to\fz}\|f\mathbf{1}_{\{x\in\Omega:\ |f(x)|>n\}}\|_{\WP}=0\r\}.$$
\end{definition}

\begin{lemma}\label{comp}
Let $\V$ be a Musielak--Orlicz function with uniformly
upper type $p_{\V}^{+}$ for some $p_{\V}^{+}\in\D$.
\begin{enumerate}
\item[{\rm (i)}] For any measurable functions $g\in\WP$ and $h\in\cwp$,
if $|g|$ is pointwise $\MP$-almost everywhere bounded by $|h|$, then $g\in\cwp$.
\item[{\rm (ii)}] If $g$, $h\in\cwp$, then, for any $c_1$, $c_2\in\cc$, $c_1g+c_2h\in\cwp$.
\item[{\rm (iii)}] If $\{g_n\}_{n\in\nn}\subset \cwp$ and
there exists a measurable function
$g$ such that $\lim_{n\to\fz}\|g_n-g\|_{\WP}=0$, then $g\in\cwp$.
\end{enumerate}
\end{lemma}
\begin{proof}
It is clear that (i) and (ii) hold true. Now we prove (iii).
For any fixed $\varepsilon\in\D$, by the condition that
$\lim_{n\to\fz}\|g_n-g\|_{\WP}=0$, we know that there exists a positive
integer $N_0$ such that, for any $n\in \nn\cap(N_0,\fz)$,
\begin{align}\label{c3}
\lf\|g_n-g\r\|_{\WP}<\varepsilon.
\end{align}
Moreover, for any fixed $n_0\in \nn\cap(N_0,\fz)$,
since $g_{n_0}\in\cwp$, we find that
there exists a positive integer $k_0$ such that
\begin{align}\label{c1}
\lf\|g_{n_0}\mathbf{1}_{\{x\in\Omega:\ |g_{n_0}(x)|>k_0\}}\r\|_{\WP}<\varepsilon.
\end{align}
Combining this and the definition of $\WP$, we conclude that
$$\sup_{\alpha\in\D}\int_{\{x\in\Omega:\ |g_{n_0}(x)|>\alpha\}
\cap \{x\in\Omega:\ |g_{n_0}(x)|>k_0\}}\V\lf(x,\frac{\alpha}{\varepsilon}\r)\,d\MP\le1.$$
From this, it follows that
\begin{align}\label{c2}
\sup_{\alpha\in(k_0,\fz)}\int_{\{x\in\Omega:\ |g_{n_0}(x)|>\alpha\}
}\V\lf(x,\frac{\alpha}{\varepsilon}\r)\,d\MP\le1.
\end{align}

On another hand, since $n_0\in \nn\cap(N_0,\fz)$, from \eqref{c3}, it follows that
\begin{align*}
\int_{\{x\in\Omega:\ |g_{n_0}(x)-g(x)|>k_0\}}\V\lf(x,\frac{k_0}{\varepsilon}\r)\,d\MP
\le \sup_{\alpha\in\D}\int_{\{x\in\Omega:\ |g_{n_0}(x)-g(x)|>\alpha\}}
\V\lf(x,\frac{\alpha}{\varepsilon}\r)\,d\MP\le1,
\end{align*}
which, together with \eqref{c2}, implies that, for any $k\in\nn\cap(2k_0,\fz)$,
\begin{align*}
&\sup_{\alpha\in\D}\int_{\{x\in\Omega:\ |g_{n_0}(x)|>\alpha\}
\cap \{x\in\Omega:\ |g_{n_0}(x)-g(x)|>k/2\}}\V\lf(x,\frac{\alpha}{\varepsilon}\r)\,d\MP\\
&\hs\le\max\lf\{\sup_{\alpha\in(0,k_0]}\int_{\{x\in\Omega:\ |g_{n_0}(x)-g(x)|>k/2\}}\V\lf(x,\frac{\alpha}{\varepsilon}\r)\,d\MP,
\sup_{\alpha\in(k_0,\fz)}\int_{\{x\in\Omega:\ |g_{n_0}(x)|>\alpha\}}
\V\lf(x,\frac{\alpha}{\varepsilon}\r)\,d\MP\r\}\\
&\hs\le\max\lf\{\int_{\{x\in\Omega:\ |g_{n_0}(x)-g(x)|>k_0\}}
\V\lf(x,\frac{k_0}{\varepsilon}\r)\,d\MP,1\r\}\le1.
\end{align*}
By this and the definition of $\WP$, we find that, for any $k\in\nn\cap(2k_0,\fz)$,
$$\lf\|g_{n_0}\mathbf{1}_{\{x\in\Omega:\ |g_{n_0}(x)-g(x)|>k/2\}}\r\|_{\WP}<\varepsilon.$$
Combining this, Remark \ref{tri}, \eqref{c3} and \eqref{c1},
we conclude that, for any $k\in\nn\cap(2k_0,\fz)$,
\begin{align*}
\lf\|g\mathbf{1}_{\{x\in\Omega:\ |g(x)|>k\}}\r\|_{\WP}
&\lesssim \lf\|g_{n_0}-g\r\|_{\WP}
+\lf\|g_{n_0}\mathbf{1}_{\{x\in\Omega:\ |g_{n_0}(x)|>k/2\}}\r\|_{\WP}\\
&\hs+\lf\|g_{n_0}\mathbf{1}_{\{x\in\Omega:\ |g_{n_0}(x)-g(x)|>k/2\}}\r\|_{\WP}
\lesssim \varepsilon.
\end{align*}
Thus, we have $\lim_{k\to\fz}\|g\mathbf{1}_{\{x\in\Omega:\ |g(x)|>k\}}\|_{\WP}=0,$
which completes the proof of (iii) and hence of Lemma \ref{comp}.
\end{proof}

\begin{remark}\label{close}
Let $\V$ be a Musielak--Orlicz function with uniformly
upper type $p_{\V}^{+}$ for some $p_{\V}^{+}\in\D$. From Lemma \ref{comp},
we deduce that $\cwp$ is a closed subspace of $\WP$.
\end{remark}

The following lemma is just \cite[Lemma 3.3(ii)]{lyj16}.

\begin{lemma}\label{l1}
Let $\V$ be a Musielak--Orlicz function satisfying Assumption \ref{a1-new}.
Then, for any $f\in\WP$ satisfying $\|f\|_{\WP}\neq0,$
$$\sup_{\alpha\in\D}\V\lf(\{x\in\Omega:\ |f(x)|>\alpha\},\frac{\alpha}{\|f\|_{\WP}}\r)=1.$$
\end{lemma}

For any measurable function $f$, let
$\rho_{\V}(f):=\sup_{\alpha\in\D}\V\lf(\{x\in\Omega:\ |f(x)|>\alpha\},\alpha\r).$

\begin{lemma}\label{lem-c}
Let $\V$ be a Musielak--Orlicz function satisfying Assumption \ref{a1-new}.
Then, for any measurable functions $\{h_n\}_{n\in\nn}$,
$\lim_{n\to\fz}\|h_n\|_{\WP}=0$ if and only if
$\lim_{n\to\fz}\rho_{\V}(h_n)=0.$
\end{lemma}
\begin{proof}
If $\lim_{n\to\fz}\|h_n\|_{\WP}=0$, then, for any fixed $\varepsilon\in(0,1)$,
there exists a positive integer $N_0\in\nn$ such that, for any $n\in\nn\cap(N_0,\fz),$
$\|h_n\|_{\WP}<\varepsilon.$
From this, Lemma \ref{l1} and the fact that $\V$ is of uniformly lower type $p_{\V}^{-}$,
we deduce that, for any $n\in\nn\cap(N_0,\fz),$
\begin{align*}
\rho_{\V}(h_n)\lesssim \lf[\|h_n\|_{\WP}\r]^{p_{\V}^{-}}\sup_{\alpha\in\D}
\int_{\{x\in\Omega:\ |h_n(x)|>\alpha\}}\V\lf(x,\frac{\alpha}{\|h_n\|_{\WP}}\r)\,d\MP
\lesssim \varepsilon^{p_{\V}^{-}}.
\end{align*}
This implies that $\lim_{n\to\fz}\rho_{\V}(h_n)=0$.

Conversely, if $\lim_{n\to\fz}\|h_n\|_{\WP}=0$ is not true, then there
exist a constant $\varepsilon_0\in(0,1)$ and a sequence $\{h_{n_k}\}_{k\in\nn}$
of measurable functions such that, for any $k\in\nn$, $\|h_{n_k}\|_{\WP}\geq \varepsilon_0$.
Combining this, Lemma \ref{l1} and the uniformly upper type $p_{\V}^{+}$
property of $\V$, we find that, for any $k\in\nn,$
$$1\le\sup_{\alpha\in\D}\V\lf(\{x\in\Omega:\ |h_{n_k}(x)|>\alpha\},
\frac{\alpha}{\varepsilon_0}\r)
\lesssim \varepsilon_0^{-p_{\V}^{+}}\sup_{\alpha\in\D}
\V\lf(\{x\in\Omega:\ |h_{n_k}(x)|>\alpha\},\alpha\r),$$
which implies that, for any $k\in\nn,$
$\rho_{\V}(h_{n_k})\gtrsim \varepsilon_0^{p_{\V}^{+}}.$
This contradicts $\lim_{n\to\fz}\rho_{\V}(h_n)=0$.
Thus, we have $\lim_{n\to\fz}\|h_n\|_{\WP}=0$, which completes
the proof of Lemma \ref{lem-c}.
\end{proof}

\begin{remark}\label{Me}
Let $\V$ be a Musielak--Orlicz function.
Since $\sup_{t\in\D}\int_{\Omega}\V(x,t)\,d\MP<\fz,$
it follows that, for any $t\in\D$, $d\widehat{\MP_t}:=\V(\cdot,t)d\,\MP$
is finite measure on $(\Omega,\mathcal{F},\mathbb{P})$.
Now we claim that, for any $F\in\cf$ and $t\in\D$,
$$\widehat{\MP_t}(F)=0 \Longleftrightarrow \MP(F)=0.$$
To show this, it suffices to prove that, for any $t\in\D$, $\widehat{\MP_t}(F)=0$ for some $F\in\cf$
implies that $\MP(F)=0$. Indeed, for any $t\in\D$,
$0=\widehat{\MP_t}(F)=\int_{F}\V(\cdot,t)\,d\MP$. From this and the fact that
$\V(\cdot,t)$ is strictly positive, we deduce that $\MP(F)=0$. This proves the above claim.
\end{remark}

We now state the following bounded convergence theorem.

\begin{theorem}\label{Thm-bct}
Let $\V$ be a Musielak--Orlicz function satisfying Assumption \ref{a1-new}.
Let $h$ be a measurable function on $\Omega$ and $\{h_n\}_{n\in\nn}\subset \WP$
a sequence of measurable functions such that $h_n(x)$ converges to $h(x)$ for almost
every $x\in\Omega$ as $n\to\fz$. If there exists a positive constant $M$ such that,
for any $n\in\nn$, $|h_n(x)|\le M$ for almost every $x\in\Omega$, then
$$\lim_{n\to\fz}\|h_n-h\|_{\WP}=0.$$
\end{theorem}
\begin{proof}
For any fixed $\varepsilon\in\D$, let
$$\delta:=\min\lf\{\lf[\frac{\varepsilon}{2C_{(p_{\V}^{-})}
\|\V(\cdot,1)\|_{L^1(\Omega)}}\r]^{1/p_{\V}^{-}},\frac12\r\},$$
here and hereafter, $C_{(p_{\V}^{-})}$ is the positive constant same as in \eqref{i1}.
For any $n\in\nn$, we have
\begin{align*}
\rho_{\V}(h_n-h)&=\sup_{\alpha\in\D}
\int_{\{x\in\Omega:\ |h_n(x)-h(x)|>\alpha\}}\V(x,\alpha)\,d\MP\\
&=\max\lf\{\sup_{\alpha\in(0,\delta]}
\int_{\{x\in\Omega:\ |h_n(x)-h(x)|>\alpha\}}\V(x,\alpha)\,d\MP
,\sup_{\alpha\in(\delta,\fz)}
\int_{\{x\in\Omega:\ |h_n(x)-h(x)|>\alpha\}}\V(x,\alpha)\,d\MP\r\}\\
&=:\max\{{\rm J}_{n,1},{\rm J}_{n,2}\}.
\end{align*}
We first estimate ${\rm J}_{n,1}$. By the uniformly
lower type $p_{\V}^{-}$ property of $\V$, we know that, for any $n\in\nn$,
\begin{align}\label{J1}
{\rm J}_{n,1}\le \int_{\Omega}\V(x,\delta)\,d\MP
\le C_{(p_{\V}^{-})}\delta^{p_{\V}^{-}}\int_{\Omega}\V(x,1)\,d\MP<\varepsilon.
\end{align}
Now we estimate ${\rm J}_{n,2}$. Since, for any $n\in\nn$,
$|h_n|$ is pointwise $\MP$-almost everywhere bounded by $M$
and $h_n$ converges $\MP$-almost everywhere to $h$ as $n\to\fz$, we know that
$|h|$ is pointwise $\MP$-almost everywhere bounded by $M$.
From this, we deduce that, for any $n\in\nn$,
\begin{align}\label{J2}
{\rm J}_{n,2}&\le
\sup_{\alpha\in(\delta,\fz)}
\int_{\{x\in\Omega:\ |h_n(x)-h(x)|>\alpha\}}\V\lf(x,|h_n(x)-h(x)|\r)\,d\MP\\
&\le \V(\{x\in\Omega:\ |h_n(x)-h(x)|>\delta\},2M).\noz
\end{align}
Moreover, there exists a measurable set $E\in\cf$ such that $\MP(E)=0$ and
$h_n\to h$ on $E$ as $n\to\fz$. From this and Remark \ref{Me}, it follows that
$\widehat{\MP_{2M}}(E)=0$ and $\widehat{\MP_{2M}}(\Omega)<\fz$. Then we have
$h_n$ converges to $h$ in measure $\widehat{\MP_{2M}}$, that is, for every
$\sigma\in\D$,
$$\lim_{n\to\fz}\V\lf(\{x\in\Omega:\ |h_n(x)-h(x)|>\sigma\},2M\r)=0.$$
Combining this and \eqref{J2}, we find that there exists a positive integer
$N_0$ such that, for any $n\in\nn\cap(N_0,\fz)$,
${\rm J}_{n,2}<\varepsilon,$
which, together with \eqref{J1}, implies that, for any $n\in\nn\cap(N_0,\fz)$,
$\rho_{\V}(h_n-h)<\varepsilon.$
By this and the arbitrariness of $\varepsilon$, we find that
$$\lim_{n\to\fz}\rho_{\V}(h_n-h)=0.$$
From this and Lemma \ref{lem-c}, it follows that
$\lim_{n\to\fz}\|h_n-h\|_{\WP}=0$, which completes the proof of Theorem \ref{Thm-bct}.
\end{proof}

Finally, we establish the following dominated convergence theorem.

\begin{theorem}\label{thm-dct}
Let $\V$ be a Musielak--Orlicz function satisfying Assumption \ref{a1-new}.
Let $\{h_n\}_{n\in\nn}$ be a sequence of measurable functions that
converges $\MP$-almost everywhere to a measurable function $h$. Suppose
that there exists a measurable function $g\in\cwp$ such that $|h_n|$
is pointwise $\MP$-almost everywhere bounded by $g$ for any $n\in\nn$.
Then $$\lim_{n\to\fz}\|h_n-h\|_{\WP}=0.$$
\end{theorem}
\begin{proof}
For any $\varepsilon\in\D$, since $g\in\cwp$, we deduce that there exists
a positive integer $N_0$ such that
\begin{align*}
\lf\|g\mathbf{1}_{\{x\in\Omega:\ |g(x)|>N_0\}}\r\|_{\WP}<\varepsilon.
\end{align*}
Combining this, Remark \ref{tri}
and the fact that $\{h_n\}_{n\in\nn}$ converges $\MP$-almost everywhere to
$h$ as $n\to\fz$, we obtain
\begin{align}\label{dc1}
\lf\|\lf(h_n-h\r)\mathbf{1}_{\{x\in\Omega:\ |g(x)|>N_0\}}\r\|_{\WP}
\le \lf\|2g\mathbf{1}_{\{x\in\Omega:\ |g(x)|>N_0\}}\r\|_{\WP}\lesssim\varepsilon.
\end{align}

On another hand, notice that $|h_n(x)|\le N_0$ for $\MP$-almost every
$x\in\{x\in\Omega:\ |g(x)|\le N_0\}$. Then, by Remark \ref{tri} and Theorem \ref{Thm-bct},
we know that there exists a positive integer $N$ such that, for any $n\in\nn\cap(N,\fz)$,
$\lf\|\lf(h_n-h\r)\mathbf{1}_{\{x\in\Omega:\ |g(x)|\le N_0\}}\r\|_{\WP}<\varepsilon.$
From this, \eqref{dc1} and Remark \ref{tri}, it follows that, for any $n\in\nn\cap(N,\fz)$,
$$\|h_n-h\|_{\WP}\lesssim \lf\|\lf(h_n-h\r)\mathbf{1}_{\{x\in\Omega:\ |g(x)|>N_0\}}\r\|_{\WP}
+\lf\|\lf(h_n-h\r)\mathbf{1}_{\{x\in\Omega:\ |g(x)|\le N_0\}}\r\|_{\WP}\lesssim\varepsilon.$$
By this and the arbitrariness of $\varepsilon$, we have $\lim_{n\to\fz}\|h_n-h\|_{\WP}=0.$
This finishes the proof of Theorem \ref{thm-dct}.
\end{proof}

\begin{remark}\label{rmk-s}
Let $\V$ be a Musielak--Orlicz function satisfying Assumption \ref{a1-new}.
We then let
$$W\mathcal{H}_{\V}^s(\Omega):=
\lf\{f=(f_n)_{n\in\zz_{+}}\in\cm:\ s(f)\in\cwp\r\},$$
$$W\mathcal{H}_{\V}^S(\Omega):=
\lf\{f=(f_n)_{n\in\zz_{+}}\in\cm:\ S(f)\in\cwp\r\}$$
and
$$W\mathcal{H}_{\V}^M(\Omega):=
\lf\{f=(f_n)_{n\in\zz_{+}}\in\cm:\ M(f)\in\cwp\r\}.$$
From Remark \ref{close} and the sublinearity of the operator $s$,
we deduce that $W\mathcal{H}_{\V}^s(\Omega)$ is a closed subspace
of $\WH$. Similarly, $W\mathcal{H}_{\V}^S(\Omega)$ and
$W\mathcal{H}_{\V}^M(\Omega)$ are the closed subspaces of
$WH_{\V}^S(\Omega)$ and $WH_{\V}^M(\Omega)$, respectively.

If $f\in W\mathcal{H}_{\V}^s(\Omega)\subset\WH$, by Theorem
\ref{thm-atom}, we have $f\in WH_{\rm at}^{\V,q,s}(\Omega)$.
Thus, there exists a sequence of triples, $\{\mu^k,a^k,\nu^k\}_{k\in\zz}$,
such that
$f=\sum_{k\in\zz}\mu^ka^k$ $\MP$-almost everywhere.
Now we claim that the sum $\sum_{k=m}^{\ell}\mu^ka^k$
converges to $f$ in $\WH$ as $m\to-\fz$ and $\ell\to\fz$.
Indeed, for any $m,\ \ell\in\mathbb{Z}$ with $m<\ell$, we have
\begin{align}\label{atom pf2}
f-\sum_{k=m}^{\ell}\mu^ka^k=\lf(f-f^{\nu^{\ell+1}}\r)+f^{\nu^m}
\quad {\rm and} \quad \lf[s\lf(f-f^{\nu^{\ell+1}}\r)\r]^2
=\lf[s\lf(f\r)\r]^2-\lf[s\lf(f^{\nu^{\ell+1}}\r)\r]^2.
\end{align}
Thus, we obtain
$s(f-f^{\nu^{\ell+1}})\le s(f)$
and $s(f^{\nu^m})\le s(f).$
From this, \eqref{atom pf2}, the fact that,
for $\MP$-almost every $x\in\Omega,$
$$\lim_{\ell\to\infty}s\lf(f-f^{\nu^{\ell+1}}\r)(x)=0,
\quad \lim_{m\to-\infty}s\lf(f^{\nu^m}\r)(x)=0$$
and Theorem \ref{thm-dct}, we deduce that
$$\lim_{\ell\to\infty}\lf\|s\lf(f-f^{\nu^{\ell+1}}\r)\r\|_{\WP}=0\quad
\mbox{and}\quad \lim_{m\to-\infty}\lf\|s\lf(f^{\nu^m}\r)\r\|_{\WP}=0.$$
Combining this, Remark \ref{tri} and the sublinearity of the operator $s$,
we complete the proof of the claim.
\end{remark}

\begin{remark}
Let $\V$ be as in Theorem \ref{thm-atomSM} and $f\in W\mathcal{H}_{\V}^S(\Omega)$
[resp., $W\mathcal{H}_{\V}^M(\Omega)$].
Analogously to Remark \ref{rmk-s}, from Theorem \ref{thm-dct}, we deduce that
the sum $\sum_{k=m}^{\ell}\mu^ka^k$ in Step $1)$ of the proof of Theorem \ref{thm-atomSM}
converges to $f$ in $WH_{\V}^M(\Omega)$ [resp., $WH_{\V}^S(\Omega)$]
as $m\to-\fz$ and $\ell\to\fz$, which may be have independent interest.
\end{remark}


\bigskip

\noindent  Guangheng Xie and Dachun Yang (Corresponding author)

\smallskip

\noindent  Laboratory of Mathematics and Complex Systems
(Ministry of Education of China),
School of Mathematical Sciences, Beijing Normal University,
Beijing 100875, People's Republic of China

\smallskip

\noindent {\it E-mails}: \texttt{guanghengxie@mail.bnu.edu.cn} (G. Xie)

\noindent\phantom{{\it E-mails}:} \texttt{dcyang@bnu.edu.cn} (D. Yang)

\medskip

\end{document}